\documentclass[12pt,reqno,a4paper]{amsart}
\usepackage{euscript,fullpage,amsmath,amssymb}
\usepackage[T1]{fontenc}
\usepackage{enumerate}
\newtheorem{theorem}{Theorem}
\newtheorem{remark}{Remark}
\newtheorem{lemma}{Lemma}
\newtheorem{proposition}{Proposition}

\renewcommand{\epsilon}{\varepsilon}

\DeclareMathOperator{\essup}{ess\,sup}
\DeclareMathOperator{\esinf}{ess\,inf}
\DeclareMathOperator{\var}{var}
\DeclareMathOperator{\osc}{osc}

\def\N{\mathbb{N}}
\def\Z{\mathbb{Z}}

\def\eps{\epsilon}
\title{Almost Sure Invariance Principle for Random Piecewise Expanding  Maps}
 \date{\today}
 \begin{document}
 \begin{abstract}
We prove a fiberwise almost sure invariance principle for random piecewise expanding transformations in one and higher dimensions using recent developments on  martingale techniques.
\end{abstract}
 \maketitle
 \begin{center}
\authors{D.\ Dragi\v cevi\' c \footnote{Department of Mathematics, University of Rijeka, Rijeka Croatia and School of Mathematics and Statistics,
University of New South Wales,
Sydney NSW 2052, Australia. E-mail: {\tt \email{ddragicevic@math.uniri.hr}}.},
G.\ Froyland\footnote{School of Mathematics and Statistics,
University of New South Wales,
Sydney NSW 2052, Australia. Email: {\tt \email{G.Froyland@unsw.edu.au }}.},
C.\   Gonz\'alez-Tokman\footnote{School of Mathematics and Physics,
The University of Queensland,
St Lucia QLD 4072,
Australia. E-mail: {\tt \email{cecilia.gt@uq.edu.au}}.},
S.\ Vaienti\footnote{Aix Marseille Universit\'e, CNRS, CPT, UMR 7332, Marseille, France. E-mail: {\tt \email{vaienti@cpt.univ-mrs.fr}}.}}

\end{center}

\section{Introduction}
The objective of this note is to prove the almost sure invariance principle (ASIP) for a large class of random dynamical systems.
The random dynamics is driven by an invertible, measure preserving transformation $\sigma$ of $(\Omega, \mathcal F, \mathbb P)$ called the base transformation.
Trajectories in the phase space $X$ are formed by concatenations $f_\omega^n:=f_{\sigma^{n-1}\omega}\circ \cdots \circ f_{\sigma\omega}\circ f_\omega$ of maps from a family of maps $f_\omega:X\to X$, $\omega\in \Omega$.
For a systematic treatment of these systems we refer to~\cite{LA}.
For sufficiently regular bounded observables $\psi_\omega:X\to \mathbb{R}$, $\omega\in\Omega$, an almost sure invariance principle guarantees that the random variables $\psi_{\sigma^n\omega}\circ f_\omega^n$ can be matched with trajectories of a Brownian motion, with the error negligible compared to the length of the trajectory.
In the present paper, we consider observables defined  on some measure space $(X,m)$ which is endowed with a notion of variation.
In particular, we consider examples where the observables are functions of bounded variation or quasi-H\"older functions on a  compact subset $X$ of $\mathbb{R}^n$.
Our setting is quite similar to that of~\cite{Buzzi}, where the maps $f_\omega$ are called random Lasota--Yorke maps.

In a more  general setting and under suitable assumptions, Kifer proved in \cite{kifer}  central limit theorems (CLT)  and laws of iterated logarithm; we will briefly compare Kifer's assumptions with ours in Remark \ref{R3} below.  In \cite[Remark 2.7]{kifer}, Kifer claimed without proof (see \cite[Remark 4.1]{kifer}) a random functional CLT, i.e. the weak invariance principle (WIP), and also a strong version of the WIP with almost sure convergence, namely the almost sure invariance principle (ASIP), referring to techniques of Philip and Stout \cite{PS}.

Other early works on the ASIP for deterministic dynamical systems go back to Field, Melbourne and T\"or\"ok \cite{FieldMelbourneTorok}, Melbourne and Nicol \cite{MN1, MN2}, and very recently to A. Korepanov \cite{KO2, Korepanov}. Another important contribution to this field, using a different
approach with respect to the aforementioned papers, is S. Gou\"ezel's article [8]. All these papers also deal  with the error term in the convergence of the process. The Gou\"ezel method was used in \cite{ANV} to get the ASIP for the stationary random dynamical systems of {\em annealed} type, in contrast to the {\em quenched} systems which are the object of this paper.

In fact we present here a proof of the ASIP for our class of random transformations,  following a method recently proposed by Cuny and Merl\`evede \cite{CM}. This method is particularly powerful when applied to  non-stationary dynamical systems; it was successfully used in \cite{HNTV} for a large class of {\em sequential} systems with some expanding features and for which  only the CLT was previously  known \cite{CR}. We stress that the $\omega$-fibered random dynamical systems discussed above are also non-stationary since we use $\omega$-dependent sample measures (see below) on the underlying probability space.

Although our method for establishing ASIP follows closely the strategy outlined in~\cite{HNTV}, the results in~\cite{HNTV} deal with a different type of systems to the ones studied in the present paper. In~\cite{HNTV} the authors consider sequential dynamical systems induced from
a sequence of transformations $(T_k)_{k\in \mathbb N}$ which are then composed as:
\begin{equation}\label{CM}
\mathcal T_n:=T_n \circ \ldots \circ T_1, \quad n\in \mathbb N.
\end{equation}
In the present work the concatenations $f_\omega^n$ are driven by the ergodic, measure-preserving transformation $\sigma$ on the base space $(\Omega, \mathcal F, \mathbb P)$: we point out that   no mixing hypotheses
are imposed on $\sigma.$
Our arguments exploit the fact that under the assumptions of our paper, the associated skew product transformation
$\tau$ (see~\eqref{eq:tau}) has a unique absolutely continuous invariant measure $\mu$ (see~\eqref{mu}), while in the context of sequential systems there is no natural notion of invariant distribution even after enlarging the space.
In particular, the probability underlying our random processes will be given by the conditional measure $\mu_{\omega}$ which exhibits the equivariance property, see Section 2.1:  this
will allow us to prove the linear growth of the variance and finally to approach the  $n^{1/4}$ rate for the ASIP error.
These are considerable improvements over corresponding results for sequential systems, where one needs very strong assumptions to ensure the growth of the variance;  see Lemma 7.1 in \cite{HNTV}.


The rate which we obtain by approximating our process with a sum of i.i.d. Gaussian variables (the content of the ASIP) is of order $n^{1/4}$ times a logarithmic correction, which is very close to the $n^{1/4}$ rate for various classes of dynamical systems, not necessarily uniformly expanding or uniformly hyperbolic. For instance the  ASIP with rate $n^{1/4+\epsilon}$ was proved in the scalar and vector cases respectively in  \cite{FieldMelbourneTorok} and \cite{GO}. A scalar ASIP with rate  $n^{1/4}$ times logarithmic corrections is the result in \cite{CM}, which inspired the present work.
It would be interesting to see if it is possible to obtain sharper estimates in the random setting, either following
 the approach of \cite{GO}, and therefore generalizing our results to vector ASIP,  or the techniques of the recent
 work by Korepanov \cite{Korepanov} which, in the context of nonuniformly expanding and nonuniformly hyperbolic transformations with exponential tails, improves the rates to $n^\epsilon$, for any $\epsilon>0$. Recent results which claim "essentially
optimal rates for slowly (polynomially) mixing deterministic dynamical systems,
such as Pomeau-Manneville intermittent maps, with H\"older continuous observables" are given in \cite{CDKM}, which contains also other references on previous results on the ASIP for different circumstances and techniques.

\section{Preliminaries and statement of the main results}

\subsection{Preliminaries}
We introduce in this section the fiber maps and the associated function spaces which we will use to form the random concatenations. We will call them {\em random expanding transformations}, or {\em random Lasota-Yorke maps}. We will refer to and use the general assumptions for these maps as proposed by Buzzi \cite{Buzzi} in order to use his results on quenched decay of correlations. However, we will strengthen a few of those assumptions with the aim  of obtaining limit theorems. Our additional conditions are similar to those called {\em Dec} and {\em Min} in the paper \cite{CR}, where they were used to establish and recover a property akin to quasi-compactness for the composition of transfer operators. \\

Let  $(\Omega, \mathcal{F}, \mathbb P)$ be  a probability space and let  $\sigma :\Omega\rightarrow \Omega$ be an  invertible $\mathbb P$-preserving transformation. We will assume that
$\mathbb P$ is ergodic. Moreover, let $(X, \mathcal B)$ be a measurable space endowed with a probability measure $m$ and a notion of a variation $\var \colon L^1(X, m) \to [0, \infty]$ which satisfies
the following conditions:
\begin{enumerate}
 \item[(V1)] $\var (th)=\lvert t\rvert \var (h)$;
 \item[(V2)] $\var (g+h)\le \var (g)+\var (h)$;
 \item[(V3)] $\lVert h\rVert_\infty \le C_{\var}(\lVert h\rVert_1+\var (h))$ for some constant $1\le C_{\var}<\infty$;
 \item[(V4)] for any $C>0$, the set  $\{h\colon X \to \mathbb R: \lVert h\rVert_1+\var (h) \le C\}$ is $L^1(X,m)$-compact;
 \item[(V5)] $\var(1_X) <\infty$, where $1_X$ denotes the function equal to $1$ on $X$;
 \item[(V6)] $\{h \colon X \to \mathbb R_+: \lVert h\rVert_1=1 \ \text{and} \ \var (h)<\infty\}$ is $L^1(X,m)$-dense in
 $\{h\colon X \to \mathbb R_+: \lVert h\rVert_1=1\}$.
 \item[(V7)] there exists $K_{\rm var}<\infty$ such that
\begin{equation}\label{varp}
 \var(gh)+\|gh\|_1 \le K_{\rm var}(\var(h)+\|h\|_1)(\var(g)+\|g\|_1), \quad \text{for every $g,h\in BV$,}
 \end{equation}
where
\[
BV=BV(X, m)=\{h\in L^1(X, m): \var(h)<\infty\};
\]
 \item[(V8)] for any $g\in L^1(X, m)$ such that $\esinf g>0$, we have $\var(1/g) \le \frac{\var (g)}{(\esinf g)^2}$.
\end{enumerate}
We recall that   $BV$ is a Banach space with respect to the norm
\[
 \lVert h\rVert_{BV}=\var (h)+\lVert h\rVert_1.
 \]
 On several occasions we will also consider the following norm
 \[
  \lVert h\rVert_{var}=\var (h)+\lVert h\rVert_\infty,
 \]
on $BV$ which (although different) is  equivalent to $\lVert \cdot \rVert_{BV}$.

Let  $f_{\omega} \colon  X \to X$, $\omega \in \Omega$ be a collection of mappings  on $X$.
The associated skew product transformation  $\tau \colon  \Omega \times X \to  \Omega \times X$ is defined  by
\begin{equation}
\label{eq:tau}
\tau(\omega, x)=( \sigma \omega, f_{\omega}(x)),
\end{equation}
where from now on we write $\sigma^k \omega$ instead of $\sigma^k(\omega)$ for each $\omega \in \Omega$ and $k\in \Z$.
 Each transformation $f_{\omega}$ induces the corresponding transfer operator $\mathcal L_{\omega}$ acting on $L^1(X,m)$ and  defined by the following duality relation
\begin{equation}\label{duality}
\int_X(\mathcal L_{\omega} \phi)\psi \, dm=\int_X \phi(\psi \circ f_{\omega})\, dm, \quad \phi \in L^1(X, m), \ \psi \in L^\infty(X, m).
\end{equation}
For each $n\in \mathbb N$ and $\omega \in \Omega$, set
\[
f_{\omega}^n=f_{\sigma^{n-1} \omega} \circ \cdots \circ f_{\omega} \quad \text{and} \quad \mathcal L_{\omega}^{(n)}=\mathcal L_{\sigma^{n-1} \omega} \circ \cdots \circ \mathcal L_{\omega}.
\]
We say that the family of maps $f_\omega$, $\omega \in \Omega$ (or the associated family of transfer operators $\mathcal L_\omega$, $\omega \in \Omega$) is \emph{uniformly good} if:
\begin{enumerate}[{(H}1)]
\item The map $(\omega, x)\mapsto (\mathcal L_\omega H(\omega, \cdot))(x)$ is $\mathbb P \times m$-measurable, i.e. measurable on the space $(\Omega \times X, \mathcal F \times \mathcal G)$ for every $\mathbb P\times m$-measurable function $H$
such that $H(\omega, \cdot)\in L^1(X,m)$ for a.e. $\omega \in \Omega$; \label{H1}
 \item There exists $C>0$ such that
 \begin{equation*}
   \|\mathcal L_\omega \phi\|_{BV} \le C\|\phi\|_{BV}
 \end{equation*}
 for $\phi\in BV$ and $\mathbb P$ a.e. $\omega \in \Omega$.\label{ULY}
\item For $\mathbb P$ a.e. $\omega \in \Omega$,
\begin{equation*}
  \sup_{n\ge 0, \lVert \phi\rVert_{BV} \le 1 } \lVert \phi \circ f_{\sigma^n \omega}\rVert_{BV} <\infty.
  \end{equation*}\label{varc}
\item
There exists $N\in \N$ such that for each $a>0$  and any sufficiently large $n\in \mathbb N$, there exists   $c>0$ such that
\begin{equation*}
\esinf  \mathcal L_\omega^{Nn} h\ge c/2 \lVert h\rVert_1, \quad \text{for every $h\in C_a$ and a.e. $\omega \in \Omega$,}
\end{equation*}
where $C_a:=\{ \phi \in BV: \phi\ge 0 \text{ and } \var(\phi)\le a\int \phi\, dm \}.$
\label{Min2}
 \item There exist $K, \lambda >0$ such that
 \begin{equation*}
\|\mathcal L_\omega^{(n)} \phi\|_{BV} \le Ke^{-\lambda n}\|\phi\|_{BV},
\end{equation*}
for $n\ge 0$, $\mathbb P$ a.e. $\omega \in \Omega$ and $\phi \in BV$ such that $\int \phi \, dm =0$.\label{DEC}
\end{enumerate}

\begin{remark}
In Sections \ref{sec:examples1} and \ref{sec:examples2} we provide explicit examples of random dynamical systems that satisfy (H\ref{H1})--(H\ref{DEC}).
Using (H\ref{H1}), (H\ref{ULY}), and (H\ref{DEC}) we can prove the existence of a unique random absolutely continuous invariant measure for $\tau$.
\end{remark}

\begin{proposition}\label{acim}
	Let $f_\omega$, $\omega \in \Omega$ be a uniformly good family of maps on $X$. Then there exist a unique measurable and nonnegative function $h\colon \Omega \times X \to \mathbb R$ with the property that
	$h_\omega:=h(\omega, \cdot) \in BV$, $\int h_\omega \, dm=1$, $\mathcal L_\omega(h_\omega)=h_{\sigma ( \omega)}$ for a.e. $\omega \in \Omega$ and
	\begin{equation}
		\label{h}
		\essup_{\omega\in\Omega} \|h_\omega\|_{BV}<\infty.
	\end{equation}
\end{proposition}
\
\begin{proof}
	Let
	\[
	Y=\bigg{\{} v\colon \Omega \times X \to \mathbb R: \text{$v$ measurable,} \  v_\omega:=v(\omega, \cdot) \in BV \ \text{and} \ \essup_{\omega \in \Omega} \lVert v_\omega\rVert_{BV}<\infty \bigg{\}}.
	\]
	Then, $Y$ is a Banach space with respect to the norm
	\[
	\lVert v\rVert' :=\essup_{\omega \in \Omega} \lVert v_\omega\rVert_{BV}.
	\]
	Moreover, let $Y_1$ be the set of all $v\in Y$ such that $\int v_\omega\, dm=1$ and $v_\omega \ge 0$ for a.e. $\omega \in \Omega$. It is easy to verify that $Y_1$ is a closed subset of $Y$ and thus a
	complete metric space. We define a map $\mathcal L \colon Y_1 \to Y_1$ by
	\[
	(\mathcal L(v))_\omega=\mathcal L_{\sigma^{-1} \omega}v_{\sigma^{-1} \omega}, \quad  \text{for $\omega \in \Omega$ and  $\ v\in Y_1$.}
	\]
The operator  $\mathcal L_\omega$ was defined in~\eqref{duality};
	note that it follows from (H\ref{ULY}) that
	\[
	\begin{split}
	\lVert \mathcal L(v)\rVert'=\essup_{\omega \in \Omega} \lVert  (\mathcal L(v))_\omega \rVert_{BV}&\le C\essup_{\omega \in \Omega} \lVert v_{\sigma^{-1} \omega}\rVert_{BV} =C\lVert v\rVert'.
	\end{split}
	\]
	Furthermore,
	\[
	\int (\mathcal L(v))_\omega\, dm=\int \mathcal L_{\sigma^{-1} \omega}v_{\sigma^{-1} \omega}\, dm=\int v_{\sigma^{-1}\omega}\, dm=1,
	\]
	for a.e. $\omega \in \Omega$.  Finally, since $v_\omega \ge 0$ for a.e. $\omega \in \Omega$ we have (using the positivity of operators $\mathcal L_\omega$) that $\mathcal L_{\sigma^{-1} \omega}v_{\sigma^{-1} \omega} \ge 0$ for
a.e. $\omega \in \Omega$.
Hence, we conclude that $\mathcal L$ is well-defined.  Similarly,
	\[
	\lVert \mathcal L(v)-\mathcal L(w)\rVert' \le C\lVert v-w\rVert', \quad \text{for $v, w\in Y_1$}
	\]
	which shows that $\mathcal L$ is continuous.
	Choose $n_0\in \mathbb  N$ such that $Ke^{-\lambda n_0}<1$. Take arbitrary $v, w\in Y_1$ and note that by~(H\ref{DEC}),
	\[
	\begin{split}
	\lVert \mathcal L^{n_0}(v) &-\mathcal L^{n_0}(w)\rVert' =\essup_{\omega \in \Omega}\lVert \mathcal L_{\sigma^{-n_0}\omega}^{(n_0)}(v_{\sigma^{-n_0}\omega}-w_{\sigma^{-n_0} \omega})\rVert_{BV}
	\\
	&\le Ke^{-\lambda n_0} \essup_{\omega \in \Omega} \lVert v_{\sigma^{-n_0}\omega}-w_{\sigma^{-n_0} \omega}\rVert_{BV}  =Ke^{-\lambda n_0} \lVert v-w\rVert'.
	\end{split}
	\]
	Hence, $\mathcal L^{n_0}$ is a contraction on $Y_1$ and thus has a unique fixed point $\tilde h\in Y_1$. Set
	\[
	h_\omega:=\frac{1}{n_0}\tilde h_\omega+ \frac{1}{n_0} \mathcal L_{\sigma^{-1} \omega} (\tilde h_{\sigma^{-1} \omega})+\ldots + \frac{1}{n_0} \mathcal L_{\sigma^{-(n_0-1)}\omega}^{(n_0-1)} (\tilde h_{\sigma^{-(n_0-1)}\omega})\quad  \text{for $\omega \in \Omega$,}
	\]
and consider the map $h\colon \Omega \times X \to \mathbb R$ given by $h(\omega, \cdot)=h_\omega$ for $\omega \in \Omega$.
	Then, $h$ is measurable, nonnegative, $\int h_\omega \, dm=1$ and a simple computation yields $\mathcal L_\omega(h_\omega)=h_{\sigma  \omega}$. Finally, by~(H\ref{ULY}) we have that
	\[
	\essup_{\omega \in \Omega} \lVert h_\omega \rVert_{BV} \le \frac{C^{n_0}-1}{n_0(C-1)}\essup_{\omega \in \Omega} \lVert  \tilde h_\omega \rVert_{BV}<\infty.
	\]
	Thus, we have established existence of $h$. The uniqueness is obvious since each $h$ satisfying the assertion of the theorem is a fixed point of $\mathcal L$ and thus also of $\mathcal L^{n_0}$ which implies that it must be unique.
\end{proof}
At this stage, we point out that we will need (H\ref{DEC}) for our later results;  we use (H\ref{DEC}) in the proof of Proposition \ref{acim} only to give a simpler existence result for the random ACIM.
With weaker control on the properties of $f_\omega$, \cite{Buzzi00,Buzzi} proved the above existence result;
in particular, these results do not require~(H\ref{DEC}) and (H\ref{ULY}) is allowed to hold with $C=C(\omega)$ such that $\log C \in L^1(\mathbb P)$.

We define a probability measure $\mu$ on $\Omega \times X$ by
\begin{equation}\label{mu}
\mu(A \times B)=\int_{A\times B} h \, d(\mathbb P\times m), \quad \text{for $A\in \mathcal F$ and $B\in \mathcal B$.}
\end{equation}
Then, it follows from Proposition~\ref{acim} that $\mu$ is invariant with respect to $\tau$.
Furthermore, $\mu$ is obviously absolutely continuous with respect to $\mathbb P \times m$.
Finally, it follows from the uniqueness in Proposition~\ref{acim} that $\mu$ is the only measure with these properties.

Let us now consider for any $\omega \in  \Omega$ the measures $\mu_{\omega}$ on the measurable space $(X, \mathcal B)$, defined by $d\mu_{\omega}=h_{\omega}dm.$
We recall here two important properties of these measures, which are together equivalent to (\ref{mu}) and its invariance, see \cite{LA}.
First, the so-called {\em equivariant property}: $f^*_{\omega}\mu_{\omega}=\mu_{\sigma \omega}$.
Second, the disintegration of $\mu$ on the marginal $\mathbb{P};$ if $A$ is any measurable set in $\mathcal F \times  \mathcal B,$ and $A_{\omega}=\{x; (\omega, x)\in A\},$ the {\em section} at $\omega,$ then $\mu(A)=\int \mu_{\omega}(A_{\omega}) d\mathbb{P}(\omega).$
The conditional (or sample) measure $\mu_{\omega}$ will constitute the probability underlying our random processes.

%
%
%

We now describe a large class of examples of good families of maps $f_\omega$, $\omega \in \Omega$.
We first show that they satisfy properties (H\ref{H1})--(H\ref{varc}); this will crucially depend on the choice of the function space.
 We then give additional requirements in order for those maps to satisfy a condition related to (H\ref{Min2}), named {\em Min} when applied to sequential systems in \cite{CR}, and condition (H\ref{DEC}), called {\em Dec} in \cite{CR}.

    \subsection{Example 1: Random Lasota-Yorke maps}
\label{sec:examples1}
Take $X=[0, 1]$, a Borel $\sigma$-algebra $\mathcal B$ on $[0, 1]$ and the Lebesgue measure $m$ on $[0, 1]$. Furthermore, let
 \[
  \var (g)=\inf_{h=g (mod \ m)}\sup_{0=s_0<s_1<\ldots <s_n=1}\sum_{k=1}^n \lvert h(s_k)-h(s_{k-1})\rvert.
 \]
It is well known that $\var$ satisfies properties (V1)--(V8) with $C_{\var},K_{\var}=1$.
For a piecewise $C^2$ $f:[0,1]\to [0,1]$, set $\delta (f)=\esinf_{x\in [0,1]} \lvert f'\rvert$ and let $N(f)$ denote the number of intervals of monoticity of $f$.
Consider now a measurable map $\omega\mapsto f_\omega$, $\omega \in \Omega$  of piecewise $C^2$ maps on $[0, 1]$ satisfying (H1) such that
\begin{equation}\label{sd}
 N:=\sup_{\omega \in \Omega} N(f_\omega)<\infty, \  \delta:=\inf_{\omega \in \Omega} \delta (f_\omega)>1, \ \mbox{ and } D:=\sup_{\omega \in \Omega}|f''_\omega|_\infty<\infty.
\end{equation}

It is proved in~\cite{Buzzi} that the family $f_\omega$, $\omega \in \Omega$ satisfies~(H2)  with
\begin{equation}\label{K}
 C=4\bigg{(} \frac{N}{\delta} \vee 1\bigg{)}\bigg{(} \frac{D}{\delta^2}\vee 1\bigg{)}\bigg{(}\frac{1}{\delta}\vee 1\bigg{)},
\end{equation}
where for any two real-valued functions $g_1$ and $g_2$, $g_1\vee g_2=\max \{g_1, g_2\}$, and (V8) has been used for the bound $\var(1/f')\leq \frac{D}{\delta^2}$.
We note that since $N<\infty$, condition (H\ref{varc}) holds.

We now discuss conditions that imply~(H\ref{Min2}).
For each $\omega \in \Omega$, let $b_\omega$ be the number of branches of $f_\omega$, so that there are essentially disjoint sub-intervals $J_{\omega,1}, \dots, J_{\omega, b_\omega}\subset I$, with $\cup_{k=1}^{b_{\omega}} J_{\omega, k}=I$, so that $f_\omega|_{J_{\omega,k}}$ is $C^2$ for each $1\leq k \leq b_\omega$. The minimal such partition $\mathcal{P}_\omega:=\{ J_{\omega,1}, \dots, J_{\omega,b_\omega} \}$ is called the \textit{regularity partition} for $f_\omega$.
We recall from classical results, e.g. \cite{LasotaYorke}, that whenever $\delta>2$, and $\esinf_{\omega \in \Omega} \min_{1\leq k \leq b_\omega} m(J_{\omega,k})>0$, there exist  $\alpha \in (0, 1)$ and $K>0$  such that
\begin{equation*}
 \var (\mathcal L_\omega \phi)\le \alpha \var(\phi)+K \lVert \phi\rVert_1, \quad \text{for $ \phi\in BV$ and a.e.  $\omega \in \Omega$.}
\end{equation*}

More generally, when $\delta<2$, one can take an iterate $N \in \N$ so that $\delta^N>2$. If the regularity partitions $\mathcal{P}^N_\omega:=\{ J^N_{1,\omega}, \dots, J^N_{\omega,b_\omega^{(N)}} \}$ corresponding to the maps $f_\omega^{(N)}$ also satisfy
$\esinf_{\omega \in \Omega} \min_{1\leq k \leq b^{(N)}_\omega} m(J^N_{\omega,k})>0$, then
 there exist $\alpha^N \in (0, 1)$ and $K^N>0$  such that
\begin{equation}\label{RLY}
 \var (\mathcal L_\omega^N \phi)\le \alpha^N \var(\phi)+K^N \lVert \phi\rVert_1, \quad \text{for $ \phi\in BV$ and a.e.  $\omega \in \Omega$.}
\end{equation}

We will from now on assume that \eqref{RLY} holds for some $N\in \N$. Iterating, it is easy to show that
\begin{equation}\label{LY}
 \var (\mathcal L_\omega^{RN} \phi) \le (\alpha^N)^R \var (\phi)+C^N \lVert \phi \rVert_1, \quad \phi\in BV, \quad \omega \in \Omega, \quad R\in \mathbb N,
\end{equation}
where $C^N=\frac{K^N}{1-\alpha^N}$.
The proof of the following lemmas are deferred to Sections \ref{sec:app1} and \ref{sec:app2}, respectively.
\begin{lemma}
\label{lem:min2}
Suppose the following uniform covering condition holds:
\begin{equation}\label{eq:UnifCov}
\text{For every subinterval } J\subset I, \exists k= k(J) \text{ s.t. for a.e. }  \omega \in \Omega, f_\omega^{(k)}(J) = I.
\end{equation}
Then, (H\ref{Min2}) holds.
\end{lemma}
\begin{lemma}
\label{lem:dec}
There exist $K, \lambda>0$ such that~(H\ref{DEC}) holds.
\end{lemma}

\subsection{Example 2: Random piecewise expanding maps}
\label{sec:examples2}
In higher dimensions, the properties (V1)-(V8) can be checked for the so-called quasi-H\"older space $\mathcal{B}_\beta$, which in particular is injected in $L^{\infty}$ (condition (V3)) and has the algebra property (V7). Originally developed by Keller \cite{keller85} for one-dimensional dynamics, we refer the reader to the Saussol paper \cite{Saussol} for a detailed presentation of that space in higher dimensions, as well as for the proof of its main properties. In particular, using the same notations as  in~\cite{Saussol}, we use the following notion of variation:
\[
 \var_{\beta}(\phi)=\sup_{0<\epsilon \le \epsilon_0}\epsilon^{-\beta}\int_{\mathbb R^n} \osc (\phi, B_\epsilon (x))\, dx,
\]
where $\epsilon_0>0$ is sufficiently small,
\[
 \osc (\phi, B_{\epsilon}(x))=\essup_{B_\epsilon(x)}\phi-\esinf_{B_\epsilon(x)}\phi,
\]
and we define the norm (we use the notation introduced in Section 2.1),
$
||\phi||_{BV}:=\var_{\beta}(\phi)+||\phi||_1,
$ which makes the set $\{\phi\in L^1(m)| \var_{\beta}(\phi)<\infty\},$ a Banach space $\mathcal{B}_{\beta}.$
In~\cite{Saussol} it is proved that this notion of variation satisfies (V1)--(V3) and (V5)--(V7) and noted that (V4) is proven in \cite{keller85} (Theorem 1.13).
We  prove here  that~(V8)  holds too. Observe that
 \[
  \begin{split}
 \osc (1/\phi, B_\epsilon (x)) &=\essup_  {B_\epsilon(x)}(1/\phi)-\esinf_{B_\epsilon(x)}(1/\phi) \\
 &=1/\esinf_{B_\epsilon(x)}(\phi)-1/\essup_{B_\epsilon(x)}(\phi)  \\
 &=\frac{\essup_{B_\epsilon(x)}(\phi)-\esinf_{B_\epsilon(x)}(\phi)}{(\essup_{B_\epsilon(x)}(\phi))(\esinf_{B_\epsilon(x)}(\phi))} \\
 &\le \frac{\osc (\phi, B_\epsilon (x))}{(\esinf_{B_\epsilon(x)}(\phi))^2}
 \le \frac{\osc (\phi, B_\epsilon (x))}{(\esinf \phi)^2},
  \end{split}
\]
which readily implies that $\var_\beta(1/\phi) \le \frac{\var_\beta (\phi)}{(\esinf \phi)^2}$.\\

We now describe the family of maps which we will endow later with a uniformly good structure;  this class has been considered in \cite[\S2]{Saussol} and \cite[Definition 2.9]{GOT}.
Let $M$ be a compact subset of $\mathbb{R}^N$
which is the closure of its non-empty interior. We take a map $f:M\to M$
and let $\mathcal{A}=\{A_i\}_{i=1}^m$ be a finite family of disjoint open sets such
that the Lebesgue measure of $M\setminus\bigcup_{i}A_i$ is zero, and there
exist open sets $\tilde{A_i}\supset\overline{A_i}$ and $C^{1+\gamma}$ maps
$f_i: \tilde{A_i}\to \mathbb{R}^N$, for some real number $0<\gamma\leq 1$
and some sufficiently small real number $\eps_1>0$ such that
\begin{enumerate}
\item $f_i(\tilde{A_i})\supset B_{\eps_1}(f(A_i))$ for each $i$, where
  $B_{\eps}(V)$ denotes a neighborhood of size $\eps$ of the set $V.$ The
  maps $f_i$ are the local extensions of $f$ to the $\tilde{A_i}.$

\item there exists a constant $C_1$ so that for each $i$ and $x,y\in f(A_i)$ with
$\mbox{dist}(x,y)\leq\eps_1$,
$$
|\det Df_i^{-1}(x)-\det Df_i^{-1}(y)|\leq C_1|\det Df_i^{-1}(x)|\mbox{dist}(x,y)^\gamma;
$$

\item there exists $s=s(f)<1$ such that $\forall x,y\in f(\tilde{A}_i) \textrm{ with } \mbox{dist}(x,y)\leq\eps_1$, we have
$$\mbox{dist}(f_i^{-1}x,f_i^{-1}y)\leq s\, \mbox{dist}(x,y);$$

\item each $\partial A_i$ is a codimension-one embedded compact
  piecewise $C^1$ submanifold and
  \begin{equation}\label{sc}s^\gamma+\frac{4s}{1-s}Z(f)\frac{\Gamma_{N-1}}{\Gamma_N}<1.\end{equation}
  where $Z(f)=\sup\limits_{x}\sum\limits_i \#\{\textrm{smooth pieces
    intersecting } \partial A_i \textrm{ containing }x\}$ and $\Gamma_N$ is
  the volume of the unit ball in $\mathbb{R}^N$ \footnote{Condition (\ref{sc}) can be considerably weakened, see \cite{Saussol}, but its right statement requires additional definitions; for smooth boundaries assumption (\ref{sc}) is perfectly adapted to our purposes.}.
\end{enumerate}
We now consider a family of maps $\{f_\omega\}_{\omega\in\Omega}$ satisfying the above conditions for a.e.\ $\omega\in\Omega$, and with uniform constants $\epsilon_1, C_1, \gamma, s,$ and $Z:=\essup_{\omega\in\Omega} Z(f_\omega)$.
The hypotheses (H\ref{H1})--(H\ref{varc}) follow as in the one-dimensional case, provided the function $\omega \mapsto f_\omega$ is measurable.
In order to satisfy hypotheses (H\ref{Min2}) and (H\ref{DEC}), as in the one-dimensional case we impose further conditions on the above class of maps.
A Lasota-Yorke inequality
 is guaranteed for each map $f_\omega$ by \cite[Lemma 4.1]{Saussol} or \cite[Proposition 3.1]{GOT}, and this ensures the  following uniform Lasota-Yorke inequality holds, for some $0<\alpha<1, K>0$,
 \begin{equation}
\label{LYmulti}
 \var_\beta (\mathcal L_\omega \phi)\le \alpha \var_\beta(\phi)+K \lVert \phi\rVert_1, \quad \text{for $ \phi\in \mathcal{B}_\beta$ and a.e.  $\omega \in \Omega$.}
 \end{equation}
As in the one-dimensional case, we can obtain an $N$-fold Lasota-Yorke inequality (of the type (\ref{RLY}) using $\var_\beta$ in place of $\var$) provided that the regularity partitions of a.e.\ $N$-fold composition of maps have positive Lebesgue measure. 
The proof of the following lemma is deferred until Appendix \ref{sect:appB}.

   \begin{lemma}\label{RPM}
    Suppose the uniform covering condition holds for the above class of maps:
     $$\mbox{for any open set $J \in M$, there exists $k=k(J)$ such that for a.e. $\omega\in \Omega$, $f^k_{\omega}(J)=M.$}$$
      Then conditions (H\ref{Min2}) and (H\ref{DEC}) are satisfied.
   \end{lemma}

\subsection{Further properties of the random ACIM}

Let $\mu_\omega$ be, as above,  the measure on $X$ given by  $d\mu_\omega=h_\omega dm$ for $\omega \in \Omega$. We have the following important consequence of~(H\ref{DEC}), which establishes the
appropriate decay of correlations result that will be used later on.
\begin{lemma}\label{lem:buzzi}
 There exists $K>0$ and $\rho \in (0, 1)$  such that
 \begin{equation}\label{buzzi}
 \bigg{\lvert} \int \mathcal L_\omega^{(n)}(\phi h_\omega)\psi\, dm -\int \phi \, d\mu_\omega \cdot \int \psi \, d\mu_{\sigma^n \omega} \bigg{\rvert} \le K\rho^n
 \lVert \psi\rVert_\infty \cdot \lVert \phi \rVert_{var} ,
\end{equation}
for $n\ge 0$, $\psi \in L^\infty (X, m)$ and $\phi \in BV(X, m)$.
\end{lemma}

\begin{proof}
We consider two cases. Assume first that $\int\phi \, d\mu_\omega=\int \phi h_\omega \, dm=0$. Then, it follows from~(H\ref{DEC}) that
\[
 \begin{split}
  \bigg{\lvert} \int \mathcal L_\omega^{(n)}&(\phi h_\omega)\psi\, dm -\int \phi \, d\mu_\omega \cdot \int \psi \, d\mu_{\sigma^n \omega} \bigg{\rvert}
  =\bigg{\lvert} \int \mathcal L_\omega^{(n)}(\phi h_\omega)\psi\, dm  \bigg{\rvert} \\
  &\le \lVert \psi \rVert_\infty \cdot \lVert \mathcal L_\omega^{(n)}(\phi h_\omega)\rVert_1
  \le \lVert \psi \rVert_\infty \cdot \lVert \mathcal L_\omega^{(n)}(\phi h_\omega)\rVert_{BV}
  \le Ke^{-\lambda n} \lVert \phi h_\omega\rVert_{BV} \cdot \lVert \psi \rVert_\infty,
 \end{split}
\]
and thus~\eqref{buzzi} follows from~\eqref{varp} and~\eqref{h}. Now we consider the case when $\int\phi \, d\mu_\omega \neq 0$. We have
\begin{align*}
  &  \bigg{\lvert} \int \mathcal L_\omega^{(n)}(\phi h_\omega)\psi\, dm -\int \phi \, d\mu_\omega \cdot \int \psi \, d\mu_{\sigma^n \omega} \bigg{\rvert} \displaybreak[0] \\
  &= \bigg{\lvert}\int  \mathcal L_\omega^{(n)}(\phi h_\omega)\psi\, dm -\int \phi h_\omega \, dm  \cdot \int \psi h_{\sigma^n \omega} \, dm \bigg{\rvert} \displaybreak[0] \\
  &\le \lVert \psi \rVert_\infty \cdot
   \int \bigg{\lvert} \bigg{(}\mathcal L_\omega^{(n)}(\phi h_\omega) -\bigg{(}\int \phi h_\omega \, dm \bigg{)}h_{\sigma^n \omega}  \bigg{)}\bigg{\rvert}\, dm  \displaybreak[0] \\
  &=\lVert \psi \rVert_\infty \cdot \bigg{\lvert} \int \phi h_\omega \, dm \bigg{\rvert} \cdot \int  \bigg{\lvert}\mathcal L_\omega^n (\Phi-h_\omega)\bigg{\rvert}\, dm \\
  &\le \lVert \psi \rVert_\infty \cdot \bigg{\lvert} \int \phi h_\omega \, dm \bigg{\rvert} \cdot \lVert \mathcal L_\omega^n (\Phi-h_\omega)\rVert_{BV},
\end{align*}
where
\[
 \phi h_\omega=\bigg{(}\int \phi h_\omega\, dm \bigg{)}\Phi.
\]
Note that $\int (\Phi-h_\omega)\, dm=0$ and thus using~(H\ref{DEC}),
\begin{align*}
  \lVert \psi \rVert_\infty \cdot \bigg{\lvert} \int \phi h_\omega \, dm \bigg{\rvert} \cdot \lVert \mathcal L_\omega^{(n)} (\Phi-h_\omega)\rVert_{BV}\displaybreak[0]
  &\le Ke^{-\lambda n}\lVert \psi\rVert_\infty \cdot \bigg{\lvert} \int \phi h_\omega \, dm \bigg{\rvert} \cdot \lVert \Phi-h_\omega\rVert_{BV}\displaybreak[0] \\
  &\le Ke^{-\lambda n}\lVert \psi\rVert_\infty \cdot \bigg{\lVert} \bigg{(}\phi-\int  \phi h_\omega \, dm \bigg{)} h_\omega \bigg{\rVert}_{BV}.
 \end{align*}
Hence, it follows from~\eqref{varp} and~\eqref{h} that
\[
 \bigg{\lvert} \int \mathcal L_\omega^{(n)}(\phi h_\omega)\psi\, dm -\int \phi \, d\mu_\omega \cdot \int \psi \, d\mu_{\sigma^n \omega} \bigg{\rvert} \le K'e^{-\lambda n}
 \lVert \psi \rVert_\infty \cdot \lVert \phi \rVert_{BV}
\]
for some $K'>0$ and thus~\eqref{buzzi} follows readily by recalling that  $\lVert \cdot \rVert_{BV}\le \lVert \cdot \rVert_{var}$.

\end{proof}
\begin{remark} We would like to emphasize that~\eqref{buzzi} is a special case of a more general decay of correlation result obtained in~\cite{Buzzi} which does not require~(H\ref{DEC}) and yields~\eqref{buzzi}\
but with $K=K(\omega)$.
\end{remark}
Finally, we prove that condition~(H\ref{Min2}) implies that we have a  uniform lower bound for $h_\omega$.
\begin{lemma}\label{LG}
We have
\begin{equation}
\label{h2}
\esinf h_\omega \ge c/2, \quad \text{for a.e.  $\omega \in \Omega$.}
\end{equation}
\end{lemma}

\begin{proof}
We first note that it follows from Proposition~\ref{acim} that there exists $a>0$ such that $h_\omega \in C_a$ for a.e. $\omega \in \Omega$. Hence, it follows from~\eqref{Min2} applied to $h=h_\omega$  that for $n$ large,
\[
\esinf h_{\sigma^{Nn} \omega} \ge c/2 \quad \text{for a.e. $\omega \in \Omega$,}
\]
which implies the desired conclusion.
\end{proof}


\begin{remark}\label{R3}
We now briefly compare our setting, results and assumptions with those in~\cite{kifer}.  In the latter,  the  space $X$ is replaced by  a foliation $\{ \Xi_{\omega} \}_{\omega \in \Omega}$.
On the fibered subset $\bf{\Xi} :=\{ (\omega, \xi) : \omega \in \Omega, \xi \in \Xi_\omega  \}$ one can consider  the skew map $\tau(\omega, \xi)=(\sigma\omega,f_{\omega}\xi)$ with the associated {\em fiber maps} $f_{\omega}:\Xi_{\omega}\rightarrow \Xi_{\sigma\omega}$. In our situation  the $\Xi_{\omega}$'s for all $\omega$ coincide with the set  $X$ and all $f_{\omega}: X\rightarrow X$ are endomorphisms of $X$ with some regularity property.
Since this situation covers the applications we have in mind (random composition of maps),
we do not treat the case of $\omega$-dependent fibers, but in principle, all the arguments in the present paper also  extend to this more general setting.
Kifer used a martingale approximation, where the martingale approximation error in \cite{kifer} is given in terms of an infinite series (see the error $g_{\omega}$ in equation (4.18) in \cite{kifer}), which appears difficult to estimate under general assumptions. Instead, our martingale approximation error term is explicitly given in terms of a finite sum (see (\ref{marterror})), and, as mentioned above, it   can be bound without difficulty in our setting.   Furthermore, Kifer  invoked a rate of mixing, but to deal with it he assumed  strong  conditions ($\phi$-mixing and $\alpha$-mixing), which are difficult to check on concrete examples of dynamical systems. We use instead quenched decay of correlations on a space of regular observables, for example, bounded variation or quasi-H\"older and $L^{\infty}$ functions (exponential decay was shown by Buzzi \cite{Buzzi}), with an addition:
the constant that scales the norm of the observable in the decay rate is independent of the noise $\omega$; we can then satisfy the hypotheses of Sprindzuk's result, see below. 


\end{remark}
\subsection{Statement of the main result}
We are now ready to state our main result. We will consider an observable $\psi \colon \Omega \times X \to \mathbb R$. Let $\psi_\omega=\psi (\omega, \cdot )$, $\omega \in \Omega$ and assume that
\begin{equation}\label{obs}
 \sup_{\omega \in \Omega} \lVert \psi_\omega \rVert_{BV} <\infty.
\end{equation}
We also introduce the centered observable
\[
\tilde \psi_\omega=\psi_\omega-\int \psi_\omega \, d\mu_\omega, \quad \omega \in \Omega
\]
and  consider the associated Birkhoff sum $\sum_{k=0}^{n-1}\tilde \psi_{\sigma^k \omega} \circ f_\omega^k,$ and the variance \begin{equation}\label{VARI}
\tau_n^2=\mathbb{E}_{\omega}\left(\sum_{k=0}^{n-1}\tilde \psi_{\sigma^k \omega} \circ f_\omega^k\right)^2.
\end{equation}
The almost sure invariance principle
is a matching of the trajectories of the dynamical system with a Brownian motion in
such a way that the error is negligible in comparison with the Birkhoff sum. Other limit
theorems such as the central limit theorem, the functional central limit theorem and
the law of the iterated logarithm  will  be  consequences (see~\cite{PS}) of our proof of the ASIP and  therefore they will hold for random Lasota-Yorke maps.

\begin{theorem}\label{main}
Let us consider the family of uniformly good random Lasota-Yorke maps. Then there exists $\Sigma^2\ge0$ such that $\lim_{n\rightarrow \infty}\frac1n \tau_n^2=\Sigma^2.$\\
Moreover one of the following cases hold:\\
(i)  $\Sigma=0$, and this is equivalent to the existence of
 $\phi \in L^2 (\Omega \times X)$ such that (co-boundary condition)
 \begin{equation}\label{cohom}
  \tilde \psi=\phi-\phi \circ \tau.
  \end{equation}
  (ii) $\Sigma^2>0$ and in this case for $\mathbb P$-a.e. $\omega \in \Omega,$ $\forall \epsilon>\frac54,$  by enlarging the  probability space $(X, \mathcal B, \mu_\omega)$ if necessary, it is possible to find a sequence $(Z_k)_k$ of independent centered (i.e. of zero mean) Gaussian random variables
such that
 \begin{equation}\label{sca}
 \sup_{1\le k\le n}\bigg{\lvert} \sum_{i=1}^k (\tilde \psi_{\sigma^i\omega} \circ f_\omega^i)-\sum_{i=1}^k Z_i \bigg{\rvert}=O(n^{1/4}\log^{\epsilon}(n)), \quad m-a.s.
\end{equation}
\end{theorem}
\begin{remark}
We notice that the statement (ii) of the  Theorem 1 also holds $\mu_{\omega}$-a.s. and for $\mathbb{P}$-a.e. $\omega$, since the measures $\mu_{\omega}$ and $m$ are equivalent by Lemma \ref{LG}.
\end{remark}
\section{Reverse martingale difference construction}
In this section we construct the reverse martingale difference and establish various useful estimates that will play an important  role in the rest of the paper. Indeed, the proof of our main result (Theorem~\ref{main}) will be obtained
as a consequence of the recent result by Cuny and Merlev\`{e}de~\cite{CM}  applied to our reverse martingale difference.

For  $\omega \in \Omega$  and $k\in \mathbb N$, let
\[
\mathcal{T}_{\omega}^k:=(f^k_{\omega})^{-1}(\mathcal B).
\]
Furthermore, for a measurable map   $\phi \colon X \to \mathbb R$  and a $\sigma$-algebra $\mathcal H$ on $X$, $\mathbb E_\omega (\phi \rvert \mathcal H)$ will denote the conditional expectation of $\phi$
with respect to $\mathcal H$ and the measure $\mu_\omega$.
 We begin with  the following technical lemma.

\begin{lemma}\label{lem:CE}
We have
\begin{equation}\label{CE}
\mathbb E_{\omega} (\phi \circ f_\omega^l \rvert \mathcal{T}_{\omega}^n)=\bigg{(}  \frac{\mathcal L_{\sigma^{l} \omega}^{(n-l)} (h_{\sigma^l \omega} \phi )}{h_{\sigma^n \omega}} \bigg{)} \circ f_\omega^n,
\end{equation}
for each $\omega \in \Omega$ and $0\le l\le n$.
\end{lemma}

\begin{proof}
We note that the right-hand side of~\eqref{CE} is measurable with respect to $\mathcal{T}_{\omega}^n$.  Take now an arbitrary $A \in  \mathcal{T}_{\omega}^n$ and write it in the form
$A=(f_\omega^n)^{-1}( B)$ for some $B\in \mathcal B$.  We have
\[
\begin{split}
\int_A \phi\circ f_\omega^l \, d\mu_\omega &=\int_X (\phi \circ f_\omega^l) \mathbf{1}_A \,  d\mu_\omega \\
&= \int_X(\phi \circ f_\omega^l) \cdot (\mathbf{1}_B \circ f_\omega^n)\, d\mu_\omega
=\int_X \phi (\mathbf{1}_B \circ f_{\sigma^l \omega}^{n-l}) \, d\mu_{\sigma^l \omega} \\
&=\int_X h_{\sigma^l \omega} \phi (\mathbf{1}_B \circ f_{\sigma^l \omega}^{n-l}) \, dm
= \int_X \mathcal L_{\sigma^{l} \omega}^{(n-l)} (h_{\sigma^l \omega} \phi ) \mathbf{1}_B \, dm \\
&= \int_X \frac{\mathcal L_{\sigma^{l} \omega}^{(n-l)} (h_{\sigma^l \omega} \phi )}{h_{\sigma^n \omega}} \mathbf{1}_B \, d\mu_{\sigma^n \omega}
=\int_X \bigg{[} \bigg{(}  \frac{\mathcal L_{\sigma^{l} \omega}^{(n-l)} (h_{\sigma^l \omega} \phi )}{h_{\sigma^n \omega}} \bigg{)} \circ f_\omega^n \bigg{]} (\mathbf{1}_B \circ f_\omega^n)\, d \mu_\omega \\
&=\int_X \bigg{[} \bigg{(}  \frac{\mathcal L_{\sigma^{l} \omega}^{(n-l)} (h_{\sigma^l \omega} \phi)}{h_{\sigma^n \omega}} \bigg{)} \circ f_\omega^n \bigg{]} \mathbf{1}_A\, d \mu_\omega
=\int_{A} \bigg{(}  \frac{\mathcal L_{\sigma^{l} \omega}^{(n-l)} (h_{\sigma^l \omega} \phi )}{h_{\sigma^n \omega}} \bigg{)} \circ f_\omega^n \, d \mu_\omega,
\end{split}
\]
which proves~\eqref{CE}.
\end{proof}
We now return to the observable $\psi_{\omega}$ introduced in (\ref{obs}) and its centered companion $\tilde \psi_\omega=\psi_\omega-\int \psi_\omega \, d\mu_\omega, \quad \omega \in \Omega.$

Set
\begin{equation}\label{MD}
M_n=\tilde \psi_{\sigma^n \omega}+G_n-G_{n+1}\circ f_{\sigma^n \omega}, \quad n\ge 0,
\end{equation}
where $G_0=0$ and
\begin{equation}\label{rec}
G_{k+1}=\frac{\mathcal L_{\sigma^k \omega}(\tilde \psi_{\sigma^k \omega} h_{\sigma^k \omega} +G_k h_{\sigma^k \omega})}{h_{\sigma^{k+1} \omega}}, \quad k\ge 0.
\end{equation}
We emphasize that $M_n$ and $G_n$ depend on $\omega$. However, in order to avoid complicating the notation, we will not make this dependence explicit.
In preparation for the next proposition we need the following elementary result.
\begin{lemma}\label{TO}
We have
\[
\mathcal L_\omega ((\psi \circ f_\omega)\phi)=\psi \mathcal L_\omega \phi, \quad \text{for $\phi \in L^1(X, m)$ and $\psi \in L^\infty(X, m)$.}
\]
\end{lemma}
\begin{proof}
For an arbitrary $z\in L^\infty (X, m)$, we have (using~\eqref{duality}) that
\[
\begin{split}
\int_X (\psi \mathcal L_\omega \phi) z\, dm &=\int_X z\psi \mathcal L_\omega \phi \, dm=\int_X (z\circ f_\omega)(\psi \circ f_\omega) \phi \, dm=\int_X \mathcal L_\omega ((\psi \circ f_\omega)\phi) z\, dm,
\end{split}
\]
which readily implies the desired conclusion.
\end{proof}

Now we prove that the sequence $(M_n \circ f_\omega^{n})_n$ is a reversed martingale (or the reversed martingale difference) with respect to the sequence of $\sigma$-algebras
$(\mathcal{T}_{\omega}^{n})_n$.
\begin{proposition}\label{lem}
We have
\[
\mathbb{E}_\omega (M_n \circ f_\omega^{n} \rvert \mathcal{T}_{\omega}^{n+1})=0.
\]
\end{proposition}

\begin{proof}
It follows from Lemma~\ref{lem:CE} that
\begin{equation}\label{nv}
\mathbb{E}_\omega (M_n \circ f_\omega^{n} \rvert \mathcal{T}_{\omega}^{n+1})= \bigg{(} \frac{\mathcal L_{\sigma^{n} \omega} (h_{\sigma^n \omega} M_n)}{h_{\sigma^{n+1} \omega}} \bigg{)} \circ f_\omega^{n+1}.
\end{equation}
Moreover, by~\eqref{MD} we have
\[
\begin{split}
\mathcal L_{\sigma^{n} \omega} (h_{\sigma^n \omega} M_n)=\mathcal L_{\sigma^{n} \omega} (h_{\sigma^n \omega} \tilde \psi_{\sigma^n \omega}+h_{\sigma^n \omega} G_n-h_{\sigma^n \omega} (G_{n+1}\circ f_{\sigma^n \omega})).
\end{split}
\]
By Lemma~\ref{TO},
\[
\mathcal L_{\sigma^{n} \omega} (h_{\sigma^n \omega} (G_{n+1}\circ f_{\sigma^n \omega}))=G_{n+1} \mathcal L_{\sigma^{n} \omega} h_{\sigma^n \omega}=G_{n+1}h_{\sigma^{n+1} \omega},
\]
and thus it follows from~\eqref{rec} that
\[
\mathcal L_{\sigma^{n} \omega} (h_{\sigma^n \omega} M_n)=0.
\]
This conclusion of the lemma now follows readily from~\eqref{nv}.
\end{proof}
We now establish several auxiliary results that will be used in the following section. These results estimate various norms of functions related to $M_n$ and $G_n$, defined in \eqref{MD} and \eqref{rec}, respectively.
\begin{lemma}\label{GN}
 We have that
 \[
  \sup_{n\ge 0} \lVert G_n\rVert_{BV} <\infty.
 \]

\end{lemma}

\begin{proof}
 By iterating~\eqref{rec}, we obtain
\begin{equation}
\label{marterror}
G_n=\frac{1}{h_{\sigma^n \omega}} \sum_{j=0}^{n-1} \mathcal L_{\sigma^j \omega}^{(n-j)}(\tilde \psi_{\sigma^j \omega} h_{\sigma^ j \omega}), \quad n\in \N.
\end{equation}
We note that
\begin{equation}\label{0515}
\int \tilde \psi_{\sigma^j \omega} h_{\sigma^ j \omega}\, dm=\int \tilde \psi_{\sigma^j \omega} \, d\mu_{\sigma^j \omega}=0,
\end{equation}
and thus it follows from~(H\ref{DEC}) that
\[
 \bigg{\lVert}\sum_{j=0}^{n-1} \mathcal L_{\sigma^j \omega}^{(n-j)}(\tilde \psi_{\sigma^j \omega} h_{\sigma^ j \omega}) \bigg{\rVert}_{BV} \le
 K\sum_{j=0}^{n-1} e^{-\lambda (n-j)} \lVert \tilde \psi_{\sigma^j \omega} h_{\sigma^ j \omega}\rVert_{BV},
\]
for each $n\in \N$ which together with~(V8), \eqref{varp}, \eqref{h}, \eqref{h2} and~\eqref{obs} implies the conclusion of the lemma.
\end{proof}

\begin{lemma}\label{Mn}
 We have that
 \[
  \sup_{n\ge 0} \lVert M_n^2\rVert_{BV} <\infty.
 \]
\end{lemma}

\begin{proof}
 In  view of~\eqref{obs}, \eqref{MD} and Lemma~\ref{GN}, it is sufficient to show that
 \[
  \sup_{n\ge 0} \lVert G_{n+1}\circ f_{\sigma^n \omega} \rVert_{BV} <\infty.
 \]
 However, this follows directly from~(H\ref{varc}) and Lemma~\ref{GN}.
\end{proof}

\begin{lemma}\label{j}
 We have that
 \[
  \sup_{n\ge 0} \lVert \mathbb E_\omega (M_n^2 \circ f_\omega^n \rvert \mathcal{T}_{\omega}^{n+1})\rVert_\infty <\infty.
 \]

\end{lemma}

\begin{proof}
 It follows from Lemma~\ref{lem:CE} that
 \[
  \mathbb E_\omega (M_n^2 \circ f_\omega^n \rvert \mathcal{T}_{\omega}^{n+1})=\bigg{(}\frac{\mathcal L_{\sigma^n \omega}(h_{\sigma^n \omega}  M_n^2)}{h_{\sigma^{n+1} \omega}}
  \bigg{)} \circ f_\omega^{n+1},
 \]
and thus, recalling~\eqref{h2},
\[
 \sup_{n\ge 0} \lVert \mathbb E_\omega (M_n^2 \circ f_\omega^n \rvert \mathcal{T}_{\omega}^{n+1})\rVert_\infty \le \frac{2}{c} \lVert \mathcal L_{\sigma^n \omega}(h_{\sigma^n \omega}  M_n^2)
 \rVert_\infty.
\]
Taking into account~\eqref{varp}, (H\ref{ULY}), \eqref{h}, Lemma~\ref{Mn} and  the fact that $\lVert \cdot \rVert_\infty \le C_{var}\lVert \cdot \rVert_{BV}$ (see (V3))
we obtain the conclusion of the lemma.
\end{proof}

\section{Sprindzuk's Theorem and consequences}
The main tool in establishing the almost sure invariance principle is the recent result by Cuny and Merlev\`{e}de (quoted in our Theorem~\ref{CMT} in Section~\ref{MR}). However, in order to verify the assumptions of that
theorem we will first need to apply the
 following classical result due to Sprindzuk~\cite{SP}.
\begin{theorem}\label{GalKosma}
 Let $(\Omega, \mathcal B, \mu)$ be a probability space and let $(f_k)_k$ be a sequence of nonnegative and measurable functions on $\Omega$. Moreover, let $(g_k)_k$
 and $(h_k)_k$ be bounded sequences of real numbers such that $0\le g_k\le h_k$. Assume that there exists $C>0$ such that
 \begin{equation}\label{gk}
  \int \bigg{(} \sum_{m<k \le n} (f_k(x)-g_k) \bigg{)}^2 \, d\mu (x) \le C \sum_{m<k \le n} h_k
 \end{equation}
for $m, n\in \mathbb N$ such that $m<n$. Then, for every $\epsilon >0$
\[
 \sum_{1\le k \le n} f_k(x)=\sum_{1\le k\le n}g_k+O(\Theta^{1/2}(n)\log^{3/2 +\epsilon}\Theta(n)),
\]
for $\mu$-a.e. $x \in \Omega$, where $\Theta(n)=\sum_{1\le k\le n}h_k$.
\end{theorem}

\begin{lemma}\label{cru}
 For each $\epsilon >0$,
 \[
   \sum_{ k=0 }^{n-1}\mathbb{E}_\omega (M_k^2 \circ f_\omega^{k} \rvert \mathcal{T}_{\omega}^{k+1})=\sum_{ k=0}^{  n-1}\mathbb{E}_\omega (M_k^2 \circ f_\omega^{k})+
   O(\Theta^{1/2}(n)\log^{3/2 +\epsilon}\Theta(n)),
 \]
for $\mu$-a.e. $\omega \in \Omega$, where
\begin{equation}\label{Theta}
 \Theta (n)=\sum_{k=0}^{n-1} (\mathbb{E}_\omega (M_k^2 \circ f_\omega^{k})+\lVert M_k^2\rVert_{var}).
\end{equation}

\end{lemma}

\begin{proof}
Fix $\omega \in \Omega$.
 We want to apply Theorem~\ref{GalKosma} to
\[
 f_k=\mathbb{E}_\omega (M_k^2 \circ f_\omega^{k} \rvert \mathcal{T}_{\omega}^{k+1}) \quad \text{and} \quad
 g_k=\mathbb{E}_\omega (M_k^2 \circ f_\omega^{k}).
\]
We have that
\begin{equation}\label{eq:SprindzucEstimate}
\begin{split}
 & \int \bigg{[}\sum_{m<k\le n}\mathbb{E}_\omega (M_k^2 \circ f_\omega^{k} \rvert \mathcal{T}_{\omega}^{k+1}) -\sum_{m<k\le n}\mathbb{E}_\omega (M_k^2 \circ f_\omega^{k})
  \bigg{]}^2 \, d\mu_\omega \ \  \\
  &=\int \bigg{(}\sum_{m<k\le n}\mathbb{E}_\omega (M_k^2 \circ f_\omega^{k} \rvert \mathcal{T}_{\omega}^{k+1})\bigg{)}^2 \, d\mu_\omega
  -\bigg{(}\sum_{m<k\le n}\mathbb{E}_\omega (M_k^2 \circ f_\omega^{k}) \bigg{)}^2  \ \  \\
  &=\sum_{m<k \le n} \int \mathbb{E}_\omega (M_k^2 \circ f_\omega^{k} \rvert \mathcal{T}_{\omega}^{k+1})^2 \, d\mu_\omega \ \  \\
  &\phantom{=} +2 \sum_{m< i<j\le n} \int \mathbb{E}_\omega(M_i^2 \circ f_\omega^{i} \rvert \mathcal{T}_{\omega}^{i+1}) \cdot \mathbb{E}_\omega(M_j^2 \circ f_\omega^{j} \rvert \mathcal{T}_{\omega}^{j+1})
 \, d\mu_\omega \ \  \\
 &\phantom{=}-\sum_{m<k\le n}\mathbb{E}_\omega (M_k^2 \circ f_\omega^{k} )^2 -2\sum_{m< i<j\le n}\mathbb{E}_\omega (M_i^2 \circ f_\omega^{i}) \cdot  \mathbb{E}_\omega (M_j^2 \circ f_\omega^{j})
 \ \  \\
 &\le \sum_{m<k \le n} \int \mathbb{E}_\omega (M_k^2 \circ f_\omega^{k} \rvert \mathcal{T}_{\omega}^{k+1})^2 \, d\mu_\omega \ \  \\
 &\phantom{\le} +2 \sum_{m< i<j\le n} \int \mathbb{E}_\omega(M_i^2 \circ f_\omega^{i} \rvert \mathcal{T}_{\omega}^{i+1}) \cdot \mathbb{E}_\omega(M_j^2 \circ f_\omega^{j} \rvert \mathcal{T}_{\omega}^{j+1})
 \, d\mu_\omega \ \  \\
 &\phantom{\le}-2\sum_{m< i<j\le n}\mathbb{E}_\omega (M_i^2 \circ f_\omega^{i}) \cdot  \mathbb{E}_\omega (M_j^2 \circ f_\omega^{j}).
 \ \  \\
  \end{split}
  \end{equation}
On the other hand, it follows from Lemma~\ref{lem:CE} that for $i<j$ we have
\[
\begin{split}
 &\int \mathbb{E}_\omega(M_i^2 \circ f_\omega^{i} \rvert \mathcal{T}_{\omega}^{i+1}) \cdot \mathbb{E}_\omega(M_j^2 \circ f_\omega^{j} \rvert \mathcal{T}_{\omega}^{j+1})
 \, d\mu_\omega \\
 &=\int \bigg{[} \bigg{(}\frac{\mathcal L_{\sigma^{i} \omega} (h_{\sigma^i \omega} M_i^2)}{h_{\sigma^{i+1} \omega}} \bigg{)} \circ f_\omega^{i+1} \bigg{]}
 \cdot \bigg{[} \bigg{(}\frac{\mathcal L_{\sigma^{j} \omega} (h_{\sigma^j \omega} M_j^2)}{h_{\sigma^{j+1} \omega}} \bigg{)} \circ f_\omega^{j+1} \bigg{]}\, d\mu_\omega \\
 &=\int \bigg{(}\frac{\mathcal L_{\sigma^{i} \omega} (h_{\sigma^i \omega} M_i^2)}{h_{\sigma^{i+1} \omega}} \bigg{)}
 \cdot \bigg{[} \bigg{(}\frac{\mathcal L_{\sigma^{j} \omega} (h_{\sigma^j \omega} M_j^2)}{h_{\sigma^{j+1} \omega}} \bigg{)} \circ f_{\sigma^{i+1} \omega}^{j-i} \bigg{]}\, d\mu_{\sigma^{i+1} \omega} \\
 &=\int \mathcal L_{\sigma^{i} \omega} (h_{\sigma^i \omega} M_i^2) \cdot \bigg{[} \bigg{(}\frac{\mathcal L_{\sigma^{j} \omega} (h_{\sigma^j \omega} M_j^2)}{h_{\sigma^{j+1} \omega}} \bigg{)} \circ f_{\sigma^{i+1} \omega}^{j-i} \bigg{]}\, dm\\
&=\int \mathcal L_{\sigma^{i} \omega}^{(j-i+1)} (h_{\sigma^i \omega} M_i^2) \cdot \bigg{(} \frac{\mathcal L_{\sigma^{j} \omega} (h_{\sigma^j \omega} M_j^2)}{h_{\sigma^{j+1} \omega}} \bigg{)} \, dm.
 \end{split}
\]
Moreover
\[
 \mathbb{E}_\omega (M_i^2 \circ f_\omega^{i}) =\int (M_i^2 \circ f_\omega^{i}) \, d\mu_\omega
 =\int M_i^2 \, d\mu_{\sigma^i \omega}
\]
and
\[
\begin{split}
 \mathbb{E}_\omega (M_j^2 \circ f_\omega^{j}) &= \int (M_j^2 \circ f_\omega^{j}) \, d\mu_\omega
 =\int M_j^2 \, d\mu_{\sigma^j \omega}
 =\int M_j^2 h_{\sigma^j \omega} \, dm \\
 &=\int \mathcal L_{\sigma^j \omega} (M_j^2 h_{\sigma^j \omega}) \, dm
 =\int \frac{\mathcal L_{\sigma^j \omega} (M_j^2 h_{\sigma^j \omega})}{h_{\sigma^{j+1} \omega}} \, d\mu_{\sigma^{j+1} \omega}.
 \end{split}
\]
Hence,
\[
 \begin{split}
& \int \mathbb{E}_\omega(M_i^2 \circ f_\omega^{i} \rvert \mathcal{T}_{\omega}^{i+1}) \cdot \mathbb{E}_\omega(M_j^2 \circ f_\omega^{j} \rvert \mathcal{T}_{\omega}^{j+1})
 \, d\mu_\omega - \mathbb{E}_\omega (M_i^2 \circ f_\omega^{i}) \cdot  \mathbb{E}_\omega (M_j^2 \circ f_\omega^{j}) \\
&=\int \mathcal L_{\sigma^{i} \omega}^{(j-i+1)} (h_{\sigma^i \omega} M_i^2) \cdot \bigg{(} \frac{\mathcal L_{\sigma^{j} \omega} (h_{\sigma^j \omega} M_j^2)}{h_{\sigma^{j+1} \omega}} \bigg{)} \, dm
-\int M_i^2 \, d\mu_{\sigma^i \omega} \cdot \int \frac{\mathcal L_{\sigma^j \omega} (M_j^2 h_{\sigma^j \omega})}{h_{\sigma^{j+1} \omega}} \, d\mu_{\sigma^{j+1} \omega}.
 \end{split}
\]
Therefore, it follows from Lemma~\ref{lem:buzzi} that
\begin{align*}
& \int \mathbb{E}_\omega(M_i^2 \circ f_\omega^{i} \rvert \mathcal{T}_{\omega}^{i+1}) \cdot \mathbb{E}_\omega(M_j^2 \circ f_\omega^{j} \rvert \mathcal{T}_{\omega}^{j+1})
 \, d\mu_\omega - \mathbb{E}_\omega (M_i^2 \circ f_\omega^{i}) \cdot  \mathbb{E}_\omega (M_j^2 \circ f_\omega^{j}) \\
 &\le K\rho^{j-i+1} \bigg{\lVert} \frac{\mathcal L_{\sigma^j \omega} (M_j^2 h_{\sigma^j \omega})}{h_{\sigma^{j+1} \omega}}\bigg{\rVert}_\infty \cdot \lVert M_i^2\rVert_{var}.
\end{align*}
Furthermore,
\[
 \int \mathbb{E}_\omega (M_k^2 \circ f_\omega^{k} \rvert \mathcal{T}_{\omega}^{k+1})^2 \, d\mu_\omega \le \lVert \mathbb{E}_\omega (M_k^2 \circ f_\omega^{k} \rvert \mathcal{T}_{\omega}^{k+1})
 \rVert_\infty \cdot \mathbb{E}_\omega (M_k^2 \circ f_\omega^{k} ).
\]
Thus, the last two inequalities combined with \eqref{eq:SprindzucEstimate} imply that
\begin{align*}
  & \int \bigg{[}\sum_{m<k\le n}\mathbb{E}_\omega (M_k^2 \circ f_\omega^{k} \rvert \mathcal{T}_{\omega}^{k+1}) -\sum_{m<k\le n}\mathbb{E}_\omega (M_k^2 \circ f_\omega^{k})
  \bigg{]}^2 \, d\mu_\omega \displaybreak[0] \\
  &\le \sum_{m<k \le n} \int \mathbb{E}_\omega (M_k^2 \circ f_\omega^{k} \rvert \mathcal{T}_{\omega}^{k+1})^2 \, d\mu_\omega \displaybreak[0] \\
 &\phantom{\le} +2 \sum_{m< i<j\le n} \int \mathbb{E}_\omega(M_i^2 \circ f_\omega^{i} \rvert \mathcal{T}_{\omega}^{i+1}) \cdot \mathbb{E}_\omega(M_j^2 \circ f_\omega^{j} \rvert \mathcal{T}_{\omega}^{j+1})
 \, d\mu_\omega \displaybreak[0] \\
 &\phantom{\le}-2\sum_{m< i<j\le n}\mathbb{E}_\omega (M_i^2 \circ f_\omega^{i}) \cdot  \mathbb{E}_\omega (M_j^2 \circ f_\omega^{j}).
 \displaybreak[0] \\
 &\le \sum_{m<k \le n}\lVert \mathbb{E}_\omega (M_k^2 \circ f_\omega^{k} \rvert \mathcal{T}_{\omega}^{k+1})
 \rVert_\infty \cdot \mathbb{E}_\omega (M_k^2 \circ f_\omega^{k} ) \displaybreak[0] \\
 &\phantom{\le}+2K\sum_{m< i<j\le n}\rho^{j-i+1} \bigg{\lVert} \frac{\mathcal L_{\sigma^j \omega} (M_j^2 h_{\sigma^j \omega})}{h_{\sigma^{j+1} \omega}}\bigg{\rVert}_\infty \cdot \lVert M_i^2\rVert_{var},
\end{align*}
which combined with~(H\ref{ULY}), \eqref{h}, \eqref{h2} and Lemmas~\ref{Mn} and~\ref{j} implies that~\eqref{gk} holds with
\[
 h_k=\mathbb{E}_\omega (M_k^2 \circ f_\omega^{k} )+\lVert M_k^2\rVert_{var}.
\]
The conclusion of the lemma now follows directly from Theorem~\ref{GalKosma}.
\end{proof}

\section{Proof of the main theorem }\label{MR}
The goal of this section is to establish the almost sure invariance principle by proving Theorem \ref{main}. It is based on the following
 theorem due to Cuny and Merlev\`{e}de.
\begin{theorem}[\cite{CM}]\label{CMT}
 Let $(X_n)_n$ be a sequence of square integrable random variables adapted to a non-increasing filtration $(\mathcal G_n)_n$. Assume that $\mathbb E(X_n\rvert \mathcal G_{n+1})=0$ a.s.,
\begin{equation}\label{varn}
v_n^2:=\sum_{k=1}^n \mathbb E(X_k^2) \to \infty \quad  \text{when $n\to \infty$}
\end{equation}
 and that $\sup_n \mathbb E(X_n^2) <\infty$. Moreover, let $(a_n)_n$ be a non-decreasing sequence of positive numbers
 such that the sequence $(a_n/v_n^2)_n$ is non-increasing, $(a_n/v_n)$ is non-decreasing and such that
:
\begin{enumerate}
 \item \begin{equation}\label{CMT1}
        \sum_{k=1}^n (\mathbb E(X_k^2\rvert \mathcal G_{k+1})-\mathbb E(X_k^2))=o(a_n) \quad a.s.;
       \end{equation}
\item \begin{equation}\label{CMT2}
       \sum_{n\ge 1} a_n^{-v}\mathbb E(\lvert X_n\rvert^{2v}) <\infty \quad \text{for some $1\le v\le 2$.}
      \end{equation}

\end{enumerate}
Then,  enlarging our probability space if necessary, it is possible to find a sequence $(Z_k)_k$ of independent centered (i.e. of  zero mean) Gaussian variables with $\mathbb E(X_k^2)=\mathbb E(Z_k^2)$
such that
\[
 \sup_{1\le k\le n}\bigg{\lvert} \sum_{i=1}^k X_i-\sum_{i=1}^k Z_i \bigg{\rvert}=o((a_n (\lvert \log (v_n^2/a_n) \rvert +\log \log a_n))^{1/2}), \quad a.s.
\]
\end{theorem}
{\em Notations}: in what follows, we write $a_n \sim b_n$ if there exists $c \in \mathbb R \setminus \{ 0\}$  such that $\lim_{n\to \infty}a_n/b_n = c$.
\begin{proof}[Proof of Theorem~\ref{main}]
Part (i) will be proved in Proposition 3 below; we now show part (ii). Let us first suppose that by using Theorem~\ref{CMT}, we could obtain the almost sure invariance principle for the sequence $(X_k)_k=(M_k \circ f_\omega^k)_k$.
Combining this with Lemma \ref{GN}, the almost sure invariance principle
for the sequence $(\tilde \psi_{\theta^k \omega} \circ f_\omega^k)_k$, stated in Theorem~\ref{main}, follows since~\eqref{MD} implies that
$$
 \sum_{k=0}^{n-1} X_k=\sum_{k=0}^{n-1} \tilde \psi_{\sigma^k \omega} \circ f_\omega^k -G_n \circ f_\omega^n.
$$
We are now left with the proof of the ASIP
for for \[ X_n=M_n  \circ f_\omega^n \quad \text{and} \quad  \mathcal G_n=\mathcal T_\omega^n, \] and we will apply   Theorem~\ref{CMT} directly to these quantities.
We note that it follows from Lemma~\ref{cru} that
\[
 \sum_{k=1}^n (\mathbb E(X_k^2\rvert \mathcal G_{k+1})-\mathbb E(X_k^2))=O(b_n),
\]
 with
\begin{equation}\label{an}
 b_n=\Theta^{1/2}(n)\log^{3/2 +\epsilon}\Theta(n),
\end{equation}
for any positive $\epsilon$ and where $\Theta(n)$ is given by~\eqref{Theta}.  On the other hand, it follows from Lemma~\ref{Mn} that $\Theta (n) \le Dn$ for some $D>0$ and every $n\in \mathbb N.$ The last part of this section will be devoted to prove that in our case \[v_n^2=\sum_{k=0}^{n-1} \mathbb E_\omega (M_k^2\circ f_\omega^k) \sim n\Sigma^2,\] where $\Sigma^2$, whose existence is ensured by Lemma \ref{var}, is assumed strictly positive in part (ii) of this theorem. We now put
$$
a_n=v_n\log^{\epsilon'}(v_n^2), \ \epsilon'\ge \frac32+\epsilon
$$
In this way and noticing that $v_n$ is increasing,  the monotonicity assumption on $a_n/v_n$ is immediately satisfied. To deal with the other condition   $\frac{a_n}{v_n^2}=\frac{\log^{\epsilon'}(v_n^2)}{v_n},$ we observe that the function $x\rightarrow \frac{\log^{\epsilon'}(x^2)}{x}, x>0$ has negative derivative for $x$ large enough depending on the value of $\epsilon'$. This implies in our situation that the monotonicity of  $\frac{a_n}{v_n^2}$ is ensured for $n$ large enough. To deal with the (finitely many)  lower values of $n$ we use Remark 2.4 in \cite{CM}, which asserts that the condition on $\frac{a_n}{v_n^2}$ can be replaced with the following one: there exists a constant $\tilde{C}$ such that for any $n\ge 1$, $\sup_{k\ge 1}(\frac{a_k}{v_k^2})\le \tilde{C} \frac{a_n}{v_n^2}:$ the easy details are left to the reader.
We have now to show that~\eqref{CMT2} holds with $v=2$.

 Since $\sup_n \lVert M_n\rVert_\infty <\infty$, we have that $\sup_n \lVert X_n\rVert_\infty <\infty$ and thus
 \[
    \sum_{n\ge 1}a_n^{-2} \mathbb E_\omega (\lvert X_n\rvert^4) \le C \sum_{n\ge 2}a_n^{-2}\sim C\sum_{n\ge 2} \frac{1}{n \log^{2\epsilon'}n}<\infty,
\]
since $2\epsilon'>1.$
We finally notice that with our choice of $a_n$ and by renaming $\epsilon$ as $\frac{1+\epsilon}{2}$,  we get the error term claimed in (\ref{sca}).

 \end{proof}
 As we said above, in the last part of this section we will prove the linear growth of the variance.
\begin{lemma}\label{var}
 There exists $\Sigma^2\ge 0$ such that
 \begin{equation}\label{variance}
  \lim_{n\to \infty} \frac 1 n \mathbb E_\omega \bigg{(}\sum_{k=0}^{n-1} \tilde \psi_{\sigma^k \omega} \circ f_\omega^k \bigg{)}^2=\Sigma^2, \quad \text{for a.e. $\omega \in \Omega$.}
 \end{equation}

\end{lemma}

\begin{proof}
 Note that
 \[
 \begin{split}
  \mathbb E_\omega \bigg{(}\sum_{k=0}^{n-1} \tilde \psi_{\sigma^k \omega} \circ f_\omega^k \bigg{)}^2 &=
  \sum_{k=0}^{n-1} \mathbb E_\omega (\tilde \psi_{\sigma^k \omega}^2\circ f_\omega^k )+2\sum_{0\le i<j\le n-1}\mathbb E_\omega ((\tilde \psi_{\sigma^i \omega} \circ f_\omega^i)
  (\tilde \psi_{\sigma^j \omega} \circ f_\omega^j)) \\
&=\sum_{k=0}^{n-1} \mathbb E_\omega (\tilde \psi_{\sigma^k \omega}^2\circ f_\omega^k )
+2\sum_{i=0}^{n-1}\sum_{j=i+1}^{n-1}\mathbb E_{\sigma^i \omega}(\tilde \psi_{\sigma^i \omega}(\tilde \psi_{\sigma^j \omega}\circ f_{\sigma^i \omega}^{j-i})).
  \end{split}
 \]
Set $g(\omega)=\mathbb E_\omega (\tilde \psi_\omega^2)$, $\omega \in \Omega$. By applying the Birkhoff's ergodic theorem for $g$ over the ergodic measure-preserving system $(\Omega, \mathcal F, \mathbb P, \sigma)$, we find that
\[
\begin{split}
\lim_{n\to \infty} \frac 1 n  \sum_{k=0}^{n-1}  \mathbb E_\omega &(\tilde \psi_{\sigma^k \omega}^2\circ f_\omega^k ) =\lim_{n\to \infty} \frac 1n \sum_{k=0}^{n-1} g(\sigma^k \omega)=\int_\Omega g(\omega)\, d\mathbb P(\omega) \\
&=\int_{\Omega} \int_X \tilde \psi(\omega, x)^2 \, d\mu_\omega(x) \, d\mathbb P(\omega)
=\int_{\Omega \times X}\tilde \psi(\omega, x)^2\, d\mu (\omega, x),
\end{split}
\]
for a.e. $\omega \in \Omega$.
Furthermore, set
\[
\Psi(\omega)=\sum_{n=1}^\infty \int_X  \tilde \psi(\omega, x)\tilde \psi(\tau^n (\omega, x))\, d\mu_\omega(x)=\sum_{n=1}^\infty \int_X\mathcal L_\omega^n(\tilde \psi_\omega h_\omega)\tilde \psi_{\sigma^n \omega}\, dm.
\]
By~\eqref{buzzi} and~\eqref{obs}, we have
\[
\lvert \Psi(\omega)\rvert  \le \sum_{n=1}^\infty \bigg{ \lvert} \int_X\mathcal L_\omega^{(n)}(\tilde \psi_\omega h_\omega)\tilde \psi_{\sigma^n \omega}\, dm \bigg{\rvert} \le \tilde K\sum_{n=1}^\infty \rho^n=\frac{\tilde K \rho}{1-\rho},
\]
for some $\tilde K>0$ and a.e. $\omega \in \Omega$. In particular, $\Psi \in L^1(\Omega)$ and thus it follows again from Birkhoff's ergodic theorem that
\begin{equation}\label{DD}
\lim_{n\to \infty}\frac{1}{n} \sum_{i=0}^{n-1}\Psi(\sigma^i \omega)=\int_\Omega \Psi(\omega)\, d\mathbb P(\omega)=\sum_{n=1}^\infty \int_{ \Omega \times X}\tilde \psi (\omega, x)\tilde
\psi (\tau^n(\omega, x))\, d\mu (\omega, x),
\end{equation}
for a.e. $\omega \in \Omega$. In order to complete the proof of the lemma, we are going to show that
\begin{equation}\label{cv}
\lim_{n\to \infty} \frac 1 n \bigg{(}\sum_{i=0}^{n-1}\sum_{j=i+1}^{n-1}\mathbb E_{\sigma^i \omega}(\tilde \psi_{\sigma^i \omega}(\tilde \psi_{\sigma^j \omega}\circ f_{\sigma^i \omega}^{j-i}))-\sum_{i=0}^{n-1}\Psi(\sigma^i \omega)\bigg{)}=0,
\end{equation}
for a.e. $\omega \in \Omega$. Using~\eqref{buzzi}, we have  that for a.e. $\omega  \in \Omega$,
\begin{align*}
& \bigg{\lvert} \sum_{i=0}^{n-1}\sum_{j=i+1}^{n-1}\mathbb E_{\sigma^i \omega}(\tilde \psi_{\sigma^i \omega}(\tilde \psi_{\sigma^j \omega}\circ f_{\sigma^i \omega}^{j-i}))-\sum_{i=0}^{n-1}\Psi(\sigma^i \omega) \bigg{\rvert} \displaybreak[0] \\
&=\bigg{\lvert} \sum_{i=0}^{n-1}\sum_{j=i+1}^{n-1}\mathbb E_{\sigma^i \omega}(\tilde \psi_{\sigma^i \omega}(\tilde \psi_{\sigma^j \omega}\circ f_{\sigma^i \omega}^{j-i}))-\sum_{i=0}^{n-1}\sum_{k=1}^\infty \mathbb E_{\sigma^i \omega}(\tilde \psi_{\sigma^i \omega}(
\tilde \psi_{\sigma^{k+i}\omega} \circ f_{\sigma^i \omega}^k)) \bigg{\rvert} \displaybreak[0] \\
&\le \sum_{i=0}^{n-1}\sum_{k=n-i}^\infty  \bigg{\lvert} \mathbb E_{\sigma^i \omega}(\tilde \psi_{\sigma^i \omega}(
\tilde \psi_{\sigma^{k+i}\omega} \circ f_{\sigma^i \omega}^k)) \bigg{\rvert}  =\sum_{i=0}^{n-1}\sum_{k=n-i}^\infty  \bigg{\lvert}\int_X \mathcal L_{\sigma^i \omega}^{(k)}(\tilde \psi_{\sigma^i \omega} h_{ \sigma^i \omega}) \tilde \psi_{\sigma^{k+i} \omega}\, dm \bigg{\rvert} \\
& \le \tilde K \sum_{i=0}^{n-1}\sum_{k=n-i}^\infty \rho^k=\tilde K \frac{\rho}{(1-\rho)^2},
\end{align*}
which readily implies~\eqref{cv}. It follows from~\eqref{DD} and~\eqref{cv} that
\[
\lim_{n \to \infty} \frac 1 n \sum_{i=0}^{n-1}\sum_{j=i+1}^{n-1}\mathbb E_{\sigma^i \omega}(\tilde \psi_{\sigma^i \omega}(\tilde \psi_{\sigma^j \omega}\circ f_{\sigma^i \omega}^{j-i}))=\sum_{n=1}^\infty \int_{ \Omega \times X}\tilde \psi (\omega, x)\tilde
\psi (\tau^n(\omega, x))\, d\mu (\omega, x)
\]
for a.e. $\omega \in \Omega$ and therefore~\eqref{variance} holds with
\begin{equation}\label{var2}
\Sigma^2=\int_{\Omega \times X}\tilde \psi(\omega, x)^2\, d\mu (\omega, x)+2\sum_{n=1}^\infty \int_{ \Omega \times X}\tilde \psi (\omega, x)\tilde
\psi (\tau^n(\omega, x))\, d\mu (\omega, x).
\end{equation}
Finally, we note that it follows readily from~\eqref{variance} that  $\Sigma^2\ge 0$ and the proof of the lemma is completed.
\end{proof}
We now present necessary and sufficient conditions under which $\Sigma^2=0$. We note that   a similar  result is  stated in~\cite[(2.10)]{kifer} with $\tilde \psi \circ \tau$ instead of $\tilde \psi$ in~\eqref{cohom}.
\begin{proposition}
 We have that $\Sigma^2=0$ if and only if there exists $\phi \in L^2 (\Omega \times X)$ such that
 \begin{equation}\label{cohom2}
  \tilde \psi=\phi-\phi \circ \tau.
 \end{equation}

\end{proposition}

\begin{proof}
 We first observe that
\[
 \begin{split}
  \int_{\Omega \times X}& \bigg{(}\sum_{k=0}^{n-1} \tilde \psi (\tau^k (\omega, x)) \bigg{)}^2\, d\mu(\omega, x) \\
&  =\sum_{k=0}^{n-1} \int_{\Omega \times X} \tilde \psi^2(\tau^k(\omega, x))
  \, d\mu(\omega, x)
+2\sum_{k=1}^{n-1} \sum_{j=0}^{k-1} \int_{\Omega \times X}\tilde \psi (\tau^k (\omega, x) \tilde \psi (\tau^j (\omega, x))\, d\mu(\omega, x) \\
  &=n\int_{\Omega \times X} \tilde \psi^2(\omega, x)\, d\mu(\omega, x)
  +2\sum_{k=1}^{n-1} \sum_{j=0}^{k-1}\int_{\Omega \times X}\tilde \psi(\omega, x)\tilde \psi (\tau^{k-j}(\omega, x))\, d\mu(\omega, x) \\
  &=n\int_{\Omega \times X} \tilde \psi^2(\omega, x)\, d\mu(\omega, x)
+2\sum_{k=1}^{n-1}(n-k)\int_{\Omega \times X} \tilde \psi(\omega, x)\tilde \psi (\tau^k (\omega, x))\, d\mu(\omega, x),
 \end{split}
\]
and thus
\begin{align*}
 & \int_{\Omega \times X}\bigg{(}\sum_{k=0}^{n-1} \tilde \psi (\tau^k (\omega, x)) \bigg{)}^2\, d\mu(\omega, x)= \displaybreak[0] \\
 &= n\bigg{(}\int_{\Omega \times X} \tilde \psi^2(\omega, x)\, d\mu(\omega, x)
  +2\sum_{k=1}^{n-1}\int_{\Omega \times X} \tilde \psi(\omega, x)\tilde \psi (\tau^k (\omega, x))\, d\mu(\omega, x) \bigg{)}\displaybreak[0] \\
  &\phantom{=}-2\sum_{k=1}^{n-1} k\int_{\Omega \times X} \tilde \psi(\omega, x)\tilde \psi (\tau^k (\omega, x))\, d\mu(\omega, x).
\end{align*}
Assume now that $\Sigma^2=0$. Then, it follows from the above equality and~\eqref{var2} that
\begin{equation}\label{f1}
 \int_{\Omega \times X}\bigg{(}\sum_{k=0}^{n-1} \tilde \psi \circ \tau^k \bigg{)}^2\, d\mu=-2n\sum_{k=n}^\infty \int_{\Omega \times X}\tilde \psi (\tilde \psi \circ \tau^k)\, d\mu
 -2\sum_{k=1}^{n-1} k\int_{\Omega \times X} \tilde \psi (\tilde \psi \circ \tau^k)\, d\mu.
\end{equation}
On the other hand, by~\eqref{buzzi} we have that $\int_{\Omega \times X} \tilde \psi (\tilde \psi \circ \tau^k)\, d\mu \to 0$ exponentially fast when $k\to \infty$ and hence, it follows
from~\eqref{f1} that the sequence $(X_n)_n$ defined by
\[
 X_n(\omega, x)=\sum_{k=0}^{n-1} \tilde \psi (\tau^k(\omega, x)), \quad \omega \in \Omega, \ x\in X
\]
is bounded in $L^2(\Omega \times X)$. Thus, it has a subsequence $(X_{n_k})_k$ which converges weakly to some $\phi \in L^2 (\Omega \times X)$. We claim that $\phi$ satisfies~\eqref{cohom2}.
Indeed, take an arbitrary $g=\mathbf{1}_{A\times B}$, where $A\in \mathcal F$ and $B\in \mathcal B$  and observe that $g\in L^2(\Omega \times X)$ and
\[\begin{split}
 \int_{\Omega \times X} g(\tilde \psi-\phi+\phi \circ \tau)&=\lim_{k\to \infty}\int_{\Omega \times X}g(\tilde \psi -X_{n_k} +X_{n_k}\circ \tau)\, d\mu \\
 &=\lim_{k\to \infty} \int_{\Omega \times X}g(\tilde \psi \circ \tau^{n_k})\, d\mu=0,
 \end{split}
\]
where in the last equality we used~\eqref{buzzi} again. Therefore, $\tilde \psi-\phi+\phi \circ \tau=0$ which readily implies~\eqref{cohom2}.

Suppose now that there exists $\phi \in L^2(\Omega \times X)$ satisfying~\eqref{cohom2}. Then,
\[
 \frac{1}{\sqrt n} \sum_{k=0}^{n-1} \tilde \psi \circ \tau^k=\frac{1}{\sqrt n} (\phi-\phi \circ \tau^n),
\]
and thus
\[
 \bigg{\lVert}  \frac{1}{\sqrt n} \sum_{k=0}^{n-1} \tilde \psi \circ \tau^k\bigg{\rVert}_{L^2(\Omega \times X)} \le \frac{2}{\sqrt{n}}\lVert \phi\rVert_{L^2(\Omega \times X)} \to 0,
\]
when $n\to \infty$. Therefore, it follows by integrating~\eqref{variance} over $\Omega$ that
\[
 \Sigma^2=\lim_{n\to \infty} \bigg{\lVert}  \frac{1}{\sqrt n} \sum_{k=0}^{n-1} \tilde \psi \circ \tau^k\bigg{\rVert}_{L^2(\Omega \times X)}^2=0.
\]
This concludes the proof of the proposition.
\end{proof}

In the rest of the paper we   assume that $\Sigma^2>0$.
We also need the following lemmas.

\begin{lemma}\label{xij}
 We have that
 \[
  \mathbb E_\omega (X_i X_j)=0, \quad \text{for $i<j$.}
 \]

\end{lemma}
\begin{proof}
By Lemma~\ref{lem}, we conclude  that
$
 \mathbb E_\omega (M_i \circ f_\omega^i \rvert \mathcal T_{\omega}^{i+1})=0
$.
Moreover, we note that $M_j \circ f_\omega^j$ is measurable with respect to $\mathcal T_{\omega}^{i+1}$ and thus
\[
 \mathbb E_\omega ((M_j \circ f_\omega^j)(M_i \circ f_\omega^i )\rvert \mathcal T_{\omega}^{i+1})=(M_j \circ f_\omega^j)\mathbb E_\omega (M_i \circ f_\omega^i \rvert \mathcal T_{\omega}^{i+1})=0,
\]
which immediately implies desired conclusion.
\end{proof}
 We now recall that $v_n^2$  is given by~\eqref{varn}.

\begin{lemma}\label{inf}
 We have that  $v_n^2  \to \infty$ as $n \to \infty$.
\end{lemma}

\begin{proof}
 We already established  from~\eqref{MD} that $ \sum_{k=0}^{n-1} X_k=\sum_{k=0}^{n-1} \tilde \psi_{\sigma^k \omega} \circ f_\omega^k -G_n \circ f_\omega^n;$ therefore
\begin{equation}\label{J}
 \bigg{(} \sum_{k=0}^{n-1} X_k \bigg{)}^2 =  \bigg{(}\sum_{k=0}^{n-1} \tilde \psi_{\sigma^k \omega} \circ f_\omega^k \bigg{)}^2 -2 (G_n \circ f_\omega^n) \bigg{(} \sum_{k=0}^{n-1} \tilde \psi_{\sigma^k \omega} \circ f_\omega^k \bigg{)}  +(G_n^2\circ f_\omega^n).
\end{equation}
By Lemma~\ref{var} and the assumption $\Sigma^2>0$,
\begin{equation}\label{y1}
 \tau_n^2:=\mathbb E_\omega \bigg{(}\sum_{k=0}^{n-1} \tilde \psi_{\sigma^k \omega} \circ f_\omega^k \bigg{)}^2 \to \infty.
\end{equation}
On the other hand, it follows from~\eqref{obs}, \eqref{J} and Lemma~\ref{GN} that
\begin{equation}\label{y}
 \mathbb E_\omega  \bigg{(} \sum_{k=0}^{n-1} X_k \bigg{)}^2 \sim \tau_n^2.
\end{equation}
By Lemma~\ref{xij} and~\eqref{y}, we have  that
\begin{equation}\label{0532}
 v_n^2=\sum_{k=0}^{n-1} \mathbb E_\omega (X_k^2)=\mathbb E_\omega  \bigg{(} \sum_{k=0}^{n-1} X_k \bigg{)}^2 \sim  \tau_n^2,
\end{equation}
which together with~\eqref{y1}  implies the desired conclusion of Lemma~\ref{inf}.
\end{proof}

\appendix
\section{Verification of Hypotheses (H\ref{Min2}) and (H\ref{DEC}) for random Lasota-Yorke maps}
\label{sec:appendix}
\subsection{Verification of Hypothesis (H\ref{Min2})}
\label{sec:app1}
\begin{lemma}\label{lem:ConeContraction}
For sufficiently large $a>0$, we have that $\mathcal L_\omega^{RN} C_a \subset C_{a/2}$, for any sufficiently large $R\in \mathbb N$ and a.e.  $\omega \in \Omega$.
 \end{lemma}
 \begin{proof}
  Choose $\phi\in C_a$. Then, it follows from~\eqref{LY} that
  \[
   \var (\mathcal L_\omega^{RN} \phi)\le (a(\alpha^N)^R +C^N)\lVert \phi\rVert_1 \le a/2 \lVert f\rVert_1,
  \]
whenever $R$ is such that $(\alpha^N)^R<1/2$ and $a\ge \frac{C^N}{1/2-(\alpha^N)^R}$.
 \end{proof}

\begin{proof}[Proof of Lemma \ref{lem:min2}]
Let us assume without loss of generality,  that $\int \phi \,dm=1$.
Following \cite{Liverani} we claim that for every $\phi \in C_a$ there exists an interval $J =[\frac{j-1}{n}, \frac{j}{n}] \subset I$ with $n=\lceil 2a\rceil, 1\leq j <n$, such that $\esinf_{x\in J} \phi \geq \frac12 \int \phi \,dm$.

Note that $\int _J  \phi \,dm \leq |J| \essup (\phi|_J)\leq |J| (\esinf (\phi|_J) + \var (\phi_{int(J)}))$.
In particular, if the claim did not hold, we would have
\[
1=\int_I \phi \,dm = \sum_{j=1}^{n} \int_{[\frac{j-1}{n}, \frac{j}{n}]} \phi \,dm < \frac12 + \frac{1}{n} \var(\phi) \leq 1,
\]
which is a contradiction. Hence, the claim holds.

Now assume \eqref{eq:UnifCov} holds. Let $\phi \in C_a$ (with $\int \phi \,dm=1$) and let $n, J$ be as in the claim above.
Let $k=\max_{1\leq j <n} k([\frac{j-1}{n}, \frac{j}{n}])$, as guaranteed by \eqref{eq:UnifCov}.
Then,
 for a.e. $\omega\in \Omega$, $f_\omega^{(k)}(J)=I$.
From the definition of $\mathcal L$ we get
\[
\esinf \mathcal L_\omega^{(k)} \phi \geq \esinf |f_\omega^{(k)}|^{-1} \esinf (\phi|_J) \geq \frac12 \esinf |f_\omega^{(k)}|^{-1}=:\alpha^*_0,
\]
where $\alpha^*_0$ is independent of $\omega$ (recall that $\essup_{x\in I, \omega \in \Omega} |f_\omega'(x)|<\infty$).

To finish the proof, let $N$ be as in \eqref{LY}, and $R$ be sufficiently large so that $NR>k$ and the conclusion of Lemma~\ref{lem:ConeContraction} holds.
Let $c= 2\alpha^*_0 \cdot \esinf_{x\in I, \omega \in \Omega}  |f_\omega^{(NR-k)}|^{-1}$. Then, for every $\phi \in C_a$,
$\esinf \mathcal L_\omega^{(NR)} \phi \geq c/2$. In addition, by Lemma~\ref{lem:ConeContraction}, $\mathcal L_\omega^{(NR)}\phi \in C_a$ and $\int \mathcal L_\omega^{(NR)}\phi \,dm=1$. Hence by induction, we conclude that for every $n\geq R$,  $\phi \in C_a$ and $\mathbb P$ a.e. $\omega \in \Omega$,
\[
\esinf \mathcal L_\omega^{Nn} \phi \ge c/2 \lVert \phi \rVert_1.
\]
\end{proof}

\subsection{Verification of Hypothesis (H\ref{DEC})}
\label{sec:app2}
We will now establish several auxiliary results that will  show that~(H\ref{Min2}) and~\eqref{LY} imply~(H\ref{DEC}).  We begin by recalling the notion of a Hilbert metric on $C_a$. For $\phi, \psi\in BV$ we write $\phi \preceq \psi$ if $\psi-\phi \in C_a \cup \{0\}$. Furthermore, for $\phi, \psi\in BV$ we define
\[
\Xi (\phi, \psi):=\sup \{\lambda \in \mathbb R^+: \lambda \phi \preceq \psi \} \quad \text{and} \quad \Upsilon(\phi, \psi):= \inf \{ \mu \in \mathbb R^+: \psi \preceq \mu \phi\},
\]
where we take $\Xi (\phi, \psi)=0$ and $\Upsilon(\phi, \psi)=\infty$ if the corresponding sets are empty. Finally, set
\[
\Theta_a (\phi, \psi)=\log \frac{\Upsilon (\phi, \psi)}{\Xi (\phi, \psi)}, \quad \text{for $\phi, \psi\in C_a$.}
\]
We recall (see~\cite{Liverani, Buzzi}) that  for $\psi\in C_{\nu a}$ for $\nu \in (0, 1)$ such that $\lVert \psi\rVert_1=1$,  we have
\begin{equation}\label{ux}
\Theta_a(g, 1) \le \log \frac{(1+\nu)(1+V)\essup \psi}{(1-\nu)(1-V)\esinf \psi}, \quad \text{where $V=\frac{\var(1_X)}{a}$.}
\end{equation}
\begin{lemma}\label{lem:17}
 Assume that $\varphi, \psi \in C_a$, $\int \phi=\int \psi=1$.
 Then,
 \begin{equation}\label{0}
  \lVert \varphi-\psi \rVert_{BV} \le 2(1+a) \Theta_a(\varphi, \psi).
 \end{equation}
\end{lemma}
\begin{proof}
Let $r,s\ge 0$ such that $r\le 1\le s$ and $r\varphi\preceq \psi\preceq s\varphi$. Note that if  such $r$ or $s$ do not exist, we have that $\Theta_a(\varphi, \psi)=\infty$ and there is nothing to prove.
Then, we have $$\lVert \psi-\varphi \rVert_1 \le \int |\psi-r\varphi|+\int (1-r)\varphi= 2(1-r).$$
Furthermore,
$$\var(\psi-\varphi)\le \var(\psi-r\varphi)+(1-r)\var(\varphi)\le a(1-r)+(1-r)a=2a(1-r).$$
The above two estimates imply that
\[\lVert \psi-\varphi \rVert_{BV}\le 2(1-r)(1+a). \]
Since $1-r\le -\log r\le \log s/r $,  we conclude the required inequality from the definition of $\Theta_a$.
\end{proof}
\begin{lemma}\label{xc}
  For any $a\ge 2\var(1_X)$, we have that $\mathcal L_\omega^{RN}$ is a contraction on $C_a$,  for any sufficiently large $R\in \mathbb N$ and a.e.  $\omega \in \Omega$.
\end{lemma}

\begin{proof}
We follow closely \cite[Lemma 2.5.]{Buzzi}.
 Let $R$ be given by Lemma~\ref{lem:ConeContraction}. We will assume that $n=R$ also satisfies~\eqref{Min2} (with respect to some $c$). Take now $\phi, \psi\in C_a$. It is sufficient to consider the case
 when $\lVert \phi\rVert_1=\lVert \psi\rVert_1=1$. Then,
 \[
  \essup \mathcal L_\omega^{RN} \phi \le \lVert \mathcal L_\omega^{RN} \phi\rVert_1+C_{\var} \var(\mathcal L_\omega^{RN} \phi) \le (1+C_{\var}a/2)\lVert \phi\rVert_1=1+C_{\var}a/2.
  \]
By~\eqref{Min2},
\[
 \esinf \mathcal L_\omega^{RN} \phi \ge c/2.
\]
Using~\eqref{ux}, we obtain that
\[
 \Theta_a (\mathcal L_\omega^{RN} \phi , 1)\le \log \frac{3/2(1+\var(1_X)/a)(1+C_{\var}a/2)}{c(1-\var(1_X)/a)/4}
\]
Since $a\ge 2\var(1_X)$, we have
\[
 \Theta_a (\mathcal L_\omega^{RN} \phi , 1)\le \log \frac{9(1+C_{\var}a/2)}{c/2}.
\]
Using triangle inequality,
\[
 \Theta_a (\mathcal L_\omega^{RN} \phi, \mathcal L_\omega^{RN} \psi)\le 2\log \frac{9(1+C_{\var}a/2)}{c/2}=:\Delta<\infty.
\]
This implies that $\mathcal L_\omega^{RN}$ is a contraction with coefficient $\kappa \in (0, 1)$ satisfying $\kappa \le \tanh (\Delta/4)$.
\end{proof}
Let $BV^0$ denote the space of all functions in $BV$ of zero mean.
\begin{lemma}\label{lem:19}
 For each $\phi\in BV^0$, there exist $K=K(\phi), \lambda=\lambda(\phi)>0$ such that
\[
 \lVert \mathcal L_\omega^n \phi\rVert_{BV} \le Ke^{-\lambda n} \lVert \phi\rVert_{BV} \quad  \text{for a.e. $\omega \in \Omega$ and $ n\in \mathbb N$.}
\]
\end{lemma}

\begin{proof}
 Without any loss of generality, we can assume that $\lVert \phi\rVert_1=2$ and thus $\lVert \phi^+\rVert_1=\lVert \phi^-\rVert_1=1$. Obviously, there exists $a\ge 2\var(1_X)$ such that
 $\phi^+, \phi^-\in C_a$. Assume that $R$ is given by previous lemma and set $M=RN$. Write $n=kM+r$ for $k\in \mathbb N_0$ and $0\le r<M$. It follows from~(H2), \eqref{0}  and Lemma~\ref{xc} that
\[
 \begin{split}
  \lVert \mathcal L_\omega^n \phi\rVert_{BV} &= \lVert \mathcal L_\omega^{kM+r} \phi\rVert_{BV} \\
  &=\lVert \mathcal L_{\sigma^{kM} \omega}^r \mathcal L_{\omega}^{kM}f\rVert_{BV} \\
  &\le 2C^r (1+a) \Theta_a(\mathcal L_{\omega}^{kM}\phi^+, \mathcal L_{\omega}^{kM} \phi^-) \\
  &\le 2C^{M-1} (1+a) \Theta_a(\mathcal L_{\omega}^{kM}\phi^+, \mathcal L_{\omega}^{kM} \phi^-) \\
  &\le K\kappa^k  \\
  &\le \frac K 2 \kappa^k \lVert \phi\rVert_1 \\
  &\le \frac K 2 (\kappa^{1/M})^n \cdot \kappa^{-r/M} \lVert \phi\rVert_{BV},
 \end{split}
\]
for some $K>0$ which readily implies the desired conclusion.
\end{proof}
Finally, we obtain~\eqref{DEC} by removing $\phi$-dependence of $K$ and $\lambda$ in Lemma~\ref{xc}.

\begin{proof}[Proof of Lemma \ref{lem:dec}]
 Let $l^1$ denote the space of all sequences $\Phi=(\phi_n)_{n\ge 1} \subset BV^0$ such that
 \[ \lVert  \Phi\rVert_1=\sum_{n\ge 1} \lVert \phi_n\rVert_{BV} <\infty.\]
Then, $(l^1, \lVert \cdot \rVert_1)$ is a Banach space.
For each $\omega \in \Omega$ and $n\in \mathbb N$, we define a linear operator $T(\omega, n) \colon BV^0 \to l^1$ by
\[
 T(\omega, n)\phi=(\mathcal L_\omega (\phi), \mathcal L_\omega^2(\phi), \ldots, \mathcal L_\omega^n(\phi), 0,0,  \ldots  ), \quad \phi\in BV^0.
\]
We note that
 $T(\omega, n)$ is bounded operator.
Indeed, it follows from~(H2) that
 \[
  \lVert T(\omega, n)\phi\rVert_1=\sum_{k=1}^n \lVert \mathcal L_\omega^k (\phi)\rVert_{BV} \le \sum_{k=1}^n C^k \lVert \phi\rVert_{BV} \le nC^n \lVert \phi\rVert_{BV}.
 \]
 Hence, $T(\omega, n)$ is bounded.
 Furthermore, note that it follows from previous lemma that
\[
 \lVert T(\omega, n)\phi\rVert_1=\sum_{k= 1}^n \lVert \mathcal L_\omega^k(\phi)\rVert_{BV} \le \frac{K(\phi)}{1-e^{-\lambda (\phi)}} \lVert \phi\rVert_{BV}=C(\phi)\lVert \phi\rVert_{BV},
\]
where
\[
C(\phi):= \frac{K(\phi)}{1-e^{-\lambda (\phi)}}.
\]
Hence, for each $\phi\in BV^0$, we have
\[
 \sup \{\lVert T(\omega, n)\phi\rVert_1: \omega \in \Omega, \, n\in \mathbb N \}<\infty.
\]
It follows from the uniform boundedness principle that there exists $L>0$ independent on $\omega$ and $n$ such that
\[
 \lVert T(\omega, n)\rVert \le L, \quad \omega \in \Omega, \, n\in \mathbb N.
\]
Hence,
\[
 \sum_{k= 1}^n \lVert \mathcal L_\omega^k(\phi)\rVert_{BV} \le L \lVert \phi\rVert_{BV}, \quad \omega \in \Omega, \, \phi\in BV^0, \, n\in \mathbb N.
\]
In particular,
\begin{equation}\label{p}
 \lVert \mathcal L_\omega^n(\phi)\rVert_{BV}  \le L \lVert \phi\rVert_{BV}, \quad \omega \in \Omega, \, \phi\in BV^0, \, n\in \mathbb N.
\end{equation}
Using~\eqref{p}, for $\omega \in \Omega$ and $1\le k\le n$,
\[
 \lVert \mathcal L_\omega^n(\phi)\rVert_{BV} =\lVert \mathcal L_{\sigma^k \omega}^{n-k} \mathcal L_\omega^k(\phi)\rVert_{BV} \le L \lVert \mathcal L_\omega^k(\phi)\rVert_{BV}.
\]
Summing over $k$,
\[
 n\lVert \mathcal L_\omega^n(\phi)\rVert_{BV} \le L\sum_{k=1}^n \lVert \mathcal L_\omega^k(\phi)\rVert_{BV} \le L^2 \lVert \phi\rVert_{BV},
\]
and thus
\[
 \lVert \mathcal L_\omega^n(\phi)\rVert_{BV} \le \frac{L^2}{n} \lVert \phi\rVert_{BV}.
\]
We conclude that there exists $N_0\in \mathbb N$ independent on $\omega$ such that
\[
  \lVert \mathcal L_\omega^{N_0}(\phi)\rVert \le \frac{1}{e} \lVert \phi\rVert, \quad \omega \in \Omega, \, \phi\in BV^0.
\]
Take now any $n\in \mathbb N$ and write it as $n=kN_0+r$, $k\in \mathbb N_0$ and $0\le r<N_0$. Using~(H2) and the inequality above,
\[
 \lVert \mathcal L_\omega^n \phi\rVert_{BV} \le \frac{C^{N_0}}{e^k} \lVert \phi\rVert_{BV}, \quad \omega \in \Omega, \, \phi\in BV^0,
\]
which readily implies~(H\ref{DEC}).
\end{proof}
\section{Verification of Hypothesis (H4) for multidimensional random piecewise expanding maps}
\label{sect:appB}
Define the cone $\mathcal{C}_a:=\{\phi\in \mathcal{B}_{\beta}; \phi\ge 0; \var_{\beta}(\phi)\le a \  \mathbb{E}_m(\phi)\}.$
The following lemma is the multidimensional version of the first part of the proof of Lemma \ref{lem:min2}.
\begin{lemma}
\label{preRPM}
For each element $\phi\in\mathcal{C}_a$, there is an open set $J$ on which $\phi$ is bounded from below by $\mathbb{E}_m(\phi)/2$.
\end{lemma}
\begin{proof}[Proof of Lemma \ref{preRPM}]
Take a  function $\phi\in \mathcal{C}_a$  with $\int_M \phi dm=1$.
Let us consider a sufficiently fine, finite partition $Q$ of $M$ into cubes (intersected with $M$).
We recall that $\epsilon_0$ is the constant entering the seminorm $\text{var}_{\beta}$. Consider $\epsilon'<\min\{ \epsilon_0, \frac{1}{2a} \}$ and assume all elements in $Q$ have diameter less than $\epsilon'$. Then, for every $x\in B_k$ we have that $B_k \subset B_{\epsilon'}(x)$. In particular, $\osc(\phi, B_{k})\leq \osc(\phi, B_{\epsilon'}(x))$ for every $x\in B_k$.
Then,
 \begin{equation}\label{ei}
 1=\int_M \phi dm= \sum_k \int_{B_k}\phi dm\le \sum_k\int_{B_k} \sup_{B_k} \phi dm \le \sum_k\int_{B_k} (\inf_{B_k} \phi + \osc(\phi, B_{k}) )    dm \leq
 $$
 $$
 \sum_k  \int_{B_k} \left( \inf_{B_k} \phi +\osc(\phi, B_{\epsilon'}(x)) \right)dm (x) \leq  \left(\sum_k  \int_{B_k} ( \inf_{B_k} \phi ) dm \right) + \int_M\osc(\phi, B_{\epsilon'}(x))dm(x).
\end{equation}
Notice that $\int_M\osc(\phi, B_{\epsilon'}(x))dm(x) \le \epsilon'\var_{\beta}(\phi)\le \epsilon' \ a \ \int_M \phi dm <\frac12$.
Hence, there exists a cube $B_j$ where the essential infimum of $f$ is at least $1/2.$ Indeed, if this were not the case, the first term on the r.h.s. of  (\ref{ei}) would be  bounded above by $1/2$, and \eqref{ei} would not be satisfied.
\end{proof}

  \begin{proof}[Proof of Lemma \ref{RPM}]
  The proof follows closely the strategy from Appendix \ref{sec:appendix}.
  We first deal with hypothesis (H\ref{Min2}).
  Lemma \ref{lem:ConeContraction} follows similarly in multidimensional setting.
  Lemma \ref{preRPM} proves the first part of Lemma \ref{lem:min2} in our multidimensional setting, and the rest of the proof of Lemma \ref{lem:min2} proceeds verbatim.
    Now we turn to Hypthesis (H\ref{DEC}).
 In the multidimensional setting, we replace lemma \ref{lem:17} with Lemma 3.8 \cite{GOT}.
  The proofs of Lemmas~\ref{xc}--\ref{lem:19} remain valid in the multidimensional setting using $\var_\beta$ instead of $\var$.
     The proof of Lemma \ref{lem:dec} then proceeds verbatim.
      We thus obtain (H\ref{DEC}) in our multidimensional setting.
   \end{proof}
   We note that results similar to Lemmas \ref{xc} and \ref{lem:19} in the multidimensional setting are contained in Lemma 3.6 and  Theorem 2.17 \cite{GOT}, respectively.

\section*{Acknowledgments}
The research of DD and GF was supported by the Australian Research Council Discovery Project DP150100017 and, in part, by the Croatian Science Foundation under the project IP-2014-09-228.
The research of CGT was supported by ARC DECRA (DE160100147).
SV was supported by the ANR-Project {\em Perturbations}, by  the Leverhulme Trust  for support thorough the Network Grant IN-2014-021, by the
project APEX Syst\`emes dynamiques: Probabilit\'es et Approximation Diophantienne
PAD funded by the R\'egion PACA (France), and by the INdAM (Istituto Naeionale di alta Matematica).
SV and CGT warmly thank the School of Mathematics and Statistics of the University of New South Wales for the kind hospitality during the preparation of this work.
SV acknowledges several discussions with M. Abdelkader and R. Aimino on the topics covered by this paper.
GF, CGT and SV thank the American Institute of Mathematics which gave them the opportunity to meet and begin this project during the Conference {\em Stochastic methods for non-equilibrium dynamical systems}, in San Jos\'e (CA, USA), in 2015.
The authors also acknowledge the South University of Science and Technology of China in Shenzhen where they met and completed this work during the Conference  {\em Statistical Properties of Nonequilibrium Dynamical Systems, in 2016.} We finally thank the referees whose clever comments and suggestions helped us to improve the paper.

\bibliographystyle{amsplain}

\end{document}